\documentclass{article}
\usepackage{graphicx, amsmath, subcaption, amsfonts, amsthm, amssymb, comment, mathtools} 
\usepackage[margin = 2.54cm]{geometry}
\usepackage{tikz}
\usetikzlibrary{
    arrows.meta, calc, positioning, patterns, backgrounds, fit
}
\definecolor{activeviolet}{RGB}{104,72,156}
\definecolor{upperblue}{RGB}{49,108,181}
\definecolor{lowerorange}{RGB}{218,124,48}
\definecolor{pairred}{RGB}{181,54,61}
\definecolor{pairgreen}{RGB}{49,128,97}
\tikzset{
    gridvertex/.style={
        circle,
        fill=black,
        inner sep=1.3pt
    },
    selectedvertex/.style={
        circle,
        fill=black,
        draw=black,
        inner sep=2.2pt
    },
    distinguished/.style={
        circle,
        fill=white,
        draw=black,
        line width=0.8pt,
        inner sep=2.1pt
    },
    regionlabel/.style={
        font=\small
    },
    pairline/.style={
        line width=1.1pt
    },
    activebox/.style={
        draw=black,
        line width=1.2pt,
        rounded corners=1pt
    },
    oldbox/.style={
        draw=black,
        dashed,
        line width=0.8pt
    },
    updatearrow/.style={
        -{Stealth[length=2mm]},
        line width=0.9pt
    },
      gridline/.style={draw=black!27,line width=0.32pt},
      activebounds/.style={draw=activeviolet,dashed,very thick},
      initialvertex/.style={circle,fill=black,draw=white,line width=0.3pt,
        inner sep=2.1pt},
      firstpair/.style={circle,fill=pairgreen,draw=white,line width=0.3pt,
        inner sep=2.2pt},
      secondpair/.style={circle,fill=upperblue,draw=white,line width=0.3pt,
        inner sep=2.2pt},
      thirdpair/.style={circle,fill=pairred,draw=white,line width=0.3pt,
        inner sep=2.2pt},
      pairendpoint/.style={circle,fill=pairred,draw=white,line width=0.3pt,
        inner sep=2.0pt},
      zetapoint/.style={circle,fill=lowerorange,draw=white,line width=0.3pt,
        inner sep=1.7pt},
      upperregion/.style={draw=upperblue,fill=upperblue!14,line width=0.8pt},
      lowerregion/.style={draw=lowerorange,fill=lowerorange!17,line width=0.8pt}
}
\newcommand{\integergrid}[2]{%
  \draw[gridline] (0,0) grid (#1,#2);
  \foreach \x in {0,...,#1}{%
    \foreach \y in {0,...,#2}{\fill[black!42] (\x,\y) circle (0.55pt);}%
  }%
}
\usepackage{mathrsfs}
\usepackage{hyperref}
\hypersetup{
    colorlinks=true,
    linkcolor=blue,     
    filecolor=magenta,      
    urlcolor=blue,           
    pdftitle={Minimal Resolving Sets in Rectangular Grid Graphs: A Complete Characterization and Enumeration},
    pdfauthor={Caroline Huang}
}
\usepackage{cleveref}

\title{Minimal Resolving Sets in Rectangular Grid Graphs: A Complete Characterization and Enumeration}
\author{Caroline Huang}
\date{September 2026}

\newtheorem{theorem}{Theorem}[section]
\newtheorem{lemma}[theorem]{Lemma}
\newtheorem{corollary}[theorem]{Corollary}

\newtheorem{definition}[theorem]{Definition}
\newtheorem{construction}[theorem]{Construction}

\DeclareMathOperator{\Res}{Res}
\usepackage[
    backend=biber,
    style=numeric,
    giveninits=true,
    sortcites=true,
    maxbibnames=99
]{biblatex}

\DeclareNameAlias{author}{family-given}
\DeclareNameAlias{editor}{family-given}

\begin{document}

\maketitle
\begin{abstract}
    Let $m,n \geq 3$. A set of vertices $S$ of the rectangular grid graph $P_m \square P_n$ is resolving if the taxicab distance vectors of the vertices of $P_m \square P_n$ with respect to $S$ are pairwise distinct. A resolving set $S$ is minimal if no proper subset of $S$ is resolving, and a minimal resolving set of cardinality $k$ is called a $k$-minimal. We extend the work of Andersen et al., who characterized $3$-minimals, established the maximum cardinality $2\min(m,n)-2$, and posed the complete characterization and enumeration of minimal resolving sets for grids as an open problem, and Adar and Epstein, who showed that $3$ is the only possible odd cardinality and that every minimal resolving set of cardinality at least $4$ can be ordered to form a sequence corresponding to a zigzag sequence. We provide a recursive construction that generates exactly the minimal resolving sets of cardinality at least $4$ for grids. We derive from the construction closed-form formulas to enumerate the $k$-minimals for all even $4 \leq k \leq 2\min(m,n) -2$. Together with the known characterizations for cardinalities $2$ and $3$ and a direct enumeration of the $3$-minimals, this yields a complete characterization and enumeration of the minimal resolving sets of rectangular grid graphs of every possible cardinality.
\end{abstract}

\section{Introduction} \label{section: introduction}

A resolving set is a collection of landmarks in which every vertex of a graph is uniquely determined by its vector of distances to the landmarks. The usual metric dimension problem asks for the fewest number of landmarks that suffice. Here, we focus on a complementary irredundancy question: in which resolving sets is every landmark essential, even when the set is not of minimum cardinality? This distinction is substantial even for rectangular grids. Their metric dimension is always $2$, whereas an inclusion-minimal resolving set may have cardinality as large as $2\min(m,n)-2$ \cite{andersen2016minimum, melter1984metric}. Thus, metric dimension describes only the smallest cardinality resolving sets and leaves a larger family of irredundant resolving sets unexplored.

For a simple, finite, connected graph $G=(V,E)$, a subset $S = \{s_1, s_2, ..., s_\ell \}$ of $V(G)$ is a resolving set of $G$ if for any two distinct $u, v \in V(G)$,
\begin{equation*}
    (d(u,s_1), d(u,s_2),\dots, d(u, s_\ell)) \neq (d(v,s_1), d(v,s_2),\dots, d(v, s_\ell)),
\end{equation*}
where $d(x,y)$ is the length of the shortest path from vertex $x$ to vertex $y$.

For a given graph $G$, the metric dimension of $G$ is the minimum cardinality of a resolving set of $G$. Metric dimension was introduced independently by Slater \cite{slater1975leaves} and Harary and Melter \cite{harary1976metric} in the 1970s, and has since been studied in relation to network discovery \cite{beerliovaetal2006} and landmark placement for robot navigation \cite{khuller1996landmarks}. 
\\\\
\textbf{Minimal Resolving Sets.}
Much of the literature on resolving sets focuses on metric dimension \cite{chartrand2003survey, chartrandpoissonzhang2000, khuller1996landmarks}, and its variants such as edge \cite{kelencedge2018}, mixed \cite{kelencmixed2017}, broadcast \cite{genesonbroadcast2022}, and fault-tolerant \cite{hernandofaulttolerant2008}. However, a resolving set may be inclusion-minimal without having minimum cardinality. A resolving set $S$ of a graph $G$ is inclusion-minimal if no proper subset of $S$ is resolving. In this paper, we refer to inclusion-minimal resolving sets as minimal resolving sets. Minimal resolving sets of cardinality $k$ are called $k$-minimals. 
\\\\
\textbf{Previous Work.}
The metric dimension of a rectangular grid is known to be $2$, and its metric bases, or $2$-minimals, have been characterized \cite{melter1984metric}. Andersen, Grigorious, and Miller investigated minimal resolving sets of grid graphs, characterizing all $3$-minimals, establishing structural restrictions on larger minimal resolving sets, and proving that $\dim^+(P_m \square P_n) = 2\min(m,n)-2$ where $\dim^+(G)$ is the maximum cardinality of a minimal resolving set of graph $G$ \cite{andersen2016minimum}. The paper leaves ``a complete characterization and enumeration of all the minimals of the grid" as future work. 

Adar and Epstein \cite{adarepstein2016alg} showed that every minimal resolving set of cardinality at least $4$ can be ordered to form a sequence corresponding to a zigzag sequence and that $3$ is the only possible odd cardinality for minimal resolving sets. However, the correspondence to zigzag sequences does not provide an if-and-only-if criterion for minimality or an enumeration of the minimal resolving sets of a grid, as not all resolving sets corresponding to a zigzag sequence are minimal. This is the gap we address in this paper.
\\\\
\textbf{Main Results.}
We solve the complete characterization and enumeration problem for every rectangular grid $P_m \square P_n$ with $m,n \geq 3$. The main idea is to convert the global requirement of distinguishing every pair of grid vertices into local, recursive requirements with a geometrically constrained family of unresolved diagonal pairs, which allows for a description of a construction that recursively generates minimal resolving sets. At each stage, two vertices are selected from explicit admissible regions, shrinking the set of unresolved vertex pairs until no unresolved pair remains. We prove both directions: every output of the construction is minimal and resolving, and every minimal resolving set of cardinality at least $4$ arises from the construction. Consequently, the construction gives not only an algorithm for generating these sets but also an if-and-only-if structural characterization from which closed enumeration formulas follow. It also gives a direct geometric explanation of the parity phenomenon proved by Adar and Epstein, that $3$ is the only possible odd cardinality of a minimal resolving set for grids.

The main results are as follows. For every $m,n\ge 3$, Construction \ref{construction} generates exactly the inclusion-minimal resolving sets of $P_m\square P_n$ of cardinality at least $4$. That is, Construction \ref{construction} generates a set of vertices if and only if it is a minimal resolving set of cardinality at least $4$. Consequently, $P_m\square P_n$ has a $k$-minimal resolving set if and only if 
\begin{equation*}
    k\in\{2,3\} \cup \{2r:2\le r\le \min(m,n)-1\},
\end{equation*}
and we obtain the following closed formula for the number of $k$-minimals $N_{k}(m,n)$ of cardinalities greater than $4$ in Theorem \ref{thm: enumeration} where $k=2r$ and $r \geq 3$:
\begin{equation*}
    N_{2r}(m,n) = \frac{1}{r}\left[(m+3r-1)\binom{m+r-2}{2r-1}
        \binom{n+r-3}{2r-2} + (n+3r-1)\binom{n+r-2}{2r-1}\binom{m+r-3}{2r-2}\right].
\end{equation*}
Combining this with our enumeration for $k=3$, we obtain the following for the number of $k$-minimals of a grid $P_m \square P_n$ in Corollary \ref{coro: finalenum}:
\begin{align*}
      N_k(m,n) = \begin{dcases}
        4 & k=2\\
        2\binom{m}{3} + (n-2)(n-1)m + 2\binom{n}{3} + (m-2)(m-1)n - 2(m  + n - 4) & k=3
        \\
        \frac{1}{2}\left[(m+5)\binom{m}{3}\binom{n-1}{2} + (n+5)\binom{n}{3} \binom{m-1}{2}\right] + 2\binom{m-2}{2} \binom{n-2}{2} & k=4\\
        \frac{1}{r}\left[(m+3r-1)\binom{m+r-2}{2r-1}
        \binom{n+r-3}{2r-2} + (n+3r-1)\binom{n+r-2}{2r-1}\binom{m+r-3}{2r-2}\right] & k = 2r\\
        0 & \text{else,}
        \end{dcases}
    \end{align*}
    where $3 \leq r \leq \min(m,n) -1$.
\\\\
\textbf{Roadmap.}
This paper provides a characterization and enumeration through a construction that produces exactly the $k$-minimals with $k \geq 4$. Section \ref{section: Notation and Terminology} standardizes notation and global conventions used in later sections. Section \ref{section: Preliminary Lemmas} reduces the process of checking whether a set of vertices is resolving to checking pairs of vertices of one particular form. Section \ref{section: Construction} presents the construction and proves that it generates exactly the minimal resolving sets of $P_m \square P_n$ of size at least $4$. An immediate corollary of the construction is that all minimal resolving sets of $P_m \square P_n$ with size at least $4$ have even cardinality. Finally, Section \ref{section: Enumeration} provides an explicit enumeration of the $k$-minimals of $P_m \square P_n$ for even $k \geq 4$.
\\\\
\textbf{Broader Significance.}
A resolving set is defined by a global condition involving all $\binom{mn}{2}$ pairs of grid vertices. Our main structural result reduces this global condition to a recursive invariant: at every stage, all relevant unresolved pairs are encoded by a single nested active rectangle and two extremal pairs. 

This provides a tractable model for working within a substantially more difficult problem. Recent work by Bergougnoux et al. \cite{bergougnouxdefrainmcinerney2025} shows that enumerating minimal resolving sets for general graphs is equivalent to hypergraph transversal enumeration, whose solvability in total-polynomial time is a longstanding open problem in algorithmic enumeration. This places our exact enumeration for rectangular grids within a broader difficult enumeration problem. For rectangular grids, the geometric rigidity of the unresolved pairs allows us to go beyond an enumeration algorithm and obtain closed formulas for enumeration. 

\section{Notation and Terminology} \label{section: Notation and Terminology}
Throughout this paper, $P_m \square P_n$ refers to a two-dimensional $m \times n$ grid graph where $m, n \geq 3$ with vertices
\begin{equation*}
    V_{m,n} := \{0, \dots, m-1\} \times \{0, \dots, n-1\}
\end{equation*}
labeled as $(i, j)$ where $i \in \{0, 1, ..., m-1\}$ and $j \in \{0, 1, ..., n-1\}$. The first coordinate is the horizontal $x$-coordinate, and the second coordinate is the vertical $y$-coordinate.

A boundary vertex of $P_m \square P_n$ has $x$-coordinate $0$ or $m-1$ or has $y$-coordinate $0$ or $n-1$. An interior vertex of $P_m \square P_n$ is a vertex $(x,y)$ with $x \in [1, m-2]$ and $y \in [1, n-2]$.

Given vertices $u,w \in V_{m,n}$, define $\Res(\{u,w\})$ as the set of all vertices that resolve $u$ and $w$. That is,
\begin{equation*}
    \Res(\{u, w\}) := \{z \in V_{m,n} \mid d(z, u) \neq d(z, w)\}.
\end{equation*}

Equivalently, a resolving set $S$ is minimal if and only if every $s \in S$ uniquely resolves some pair of vertices. That is,
\begin{equation*}
    S \text{ is minimal} \iff \forall s \in S, \exists u, w \in V_{m,n} \text{ such that } u \neq w \text{ and } S \cap \Res(\{u, w\}) = \{s\}.
\end{equation*}
We say that a set of vertices $M$ breaks minimality if some vertex in $M$ does not uniquely resolve any pair of vertices. This implies that no superset of $M$ can form a minimal resolving set.

For integers $A,B \in [0, m-1]$ and $C,D \in [0, n-1]$, define 
\begin{align*}
    \mathcal{Z}(A,B,C,D) = \{\{(a,b),(a+t,b+t)\} \mid t\in\mathbb{Z}_{>0}; a,a+t\in[A,B]_{\mathbb{Z}}; b,b+t\in[C,D]_{\mathbb{Z}}\},\\
    \mathcal{R}(A,B,C,D) = \{(x,y) \mid x \in [A,B]_{\mathbb{Z}}, y \in [C,D]_{\mathbb{Z}}\},
\end{align*}
where $[A,B]_{\mathbb{Z}} = [A,B] \cap \mathbb{Z}$. 
\\
\begin{figure}[hbt!]
  \centering
  \begin{tikzpicture}[scale=0.6,font=\small]
    % The rectangle is drawn first so the grid remains visible.
    \fill[upperblue!10] (0.82,0.82) rectangle (5.18,4.18);
    \integergrid{6}{5}
    \draw[upperblue,very thick,rounded corners=1pt]
      (0.82,0.82) rectangle (5.18,4.18);

    % Three representative elements of Z(A,B,C,D).
    \draw[pairred,very thick] (1,1)--(3,3);
    \draw[lowerorange,very thick] (2,1)--(5,4);
    \draw[pairgreen,very thick] (3,2)--(5,4);
    \foreach \P in {(1,1),(3,3)}
      \node[pairendpoint] at \P {};
    \foreach \P in {(2,1),(5,4)}
      \node[circle,fill=lowerorange,draw=white,line width=0.3pt,
        inner sep=2pt] at \P {};
    \foreach \P in {(3,2),(5,4)}
      \node[circle,fill=pairgreen,draw=white,line width=0.3pt,
        inner sep=1.45pt] at \P {};

    \node[below=4pt] at (1,0) {$A$};
    \node[below=4pt] at (5,0) {$B$};
    \node[left=4pt] at (0,1) {$C$};
    \node[left=4pt] at (0,4) {$D$};
    \node[upperblue,fill=white,inner sep=2pt] at (3,4.55)
      {$R(A,B,C,D)$};
    \node[align=left,anchor=west] at (6.55,3.55)
      {$R(A,B,C,D)$ is a set\\of grid vertices.};
    \node[align=left,anchor=west] at (6.55,1.75)
      {$Z(A,B,C,D)$ is a set of\\slope-$1$ unordered pairs.};
    \draw[-{Stealth[length=2mm]},pairred]
      (6.45,1.65) to[bend right=10] (4.15,2.85);
  \end{tikzpicture}
  \caption{$R(A,B,C,D)$ and $Z(A,B,C,D)$ are shown above bounded inclusively by $x=A, x=B, y=C$, and $y=D$. The blue
  rectangle indicates the vertex set $R(A,B,C,D)$; the three colored
  segments represent sample unordered pairs
  $\{(a,b),(a+t,b+t)\}\in Z(A,B,C,D)$.}
  \label{fig:R-versus-Z}
\end{figure}
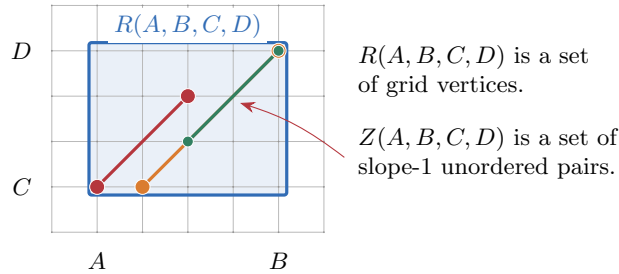

We use the convention that $[A,B] \cap \mathbb{Z} = \varnothing$ when $A > B$ and $[C,D] \cap \mathbb{Z} = \varnothing$ when $C > D$.
These definitions hold even when $A\geq B$ or $C\geq D$. Since all coordinates are integral, $\mathcal{Z}(A,B, C,D)\neq\varnothing \iff A<B \text{ and } C<D$ and $\mathcal{R}(A,B,C,D) \neq \varnothing \iff A \leq B \text{ and } C \leq D$.

\section{Reduction to a Class of Vertex Pairs} \label{section: Preliminary Lemmas}

By Proposition $2$ in \cite{andersen2016minimum}, all $k$-minimals contain two opposite-side boundary vertices. The following lemmas characterize which pairs of vertices in the grid remain unresolved once such two boundary vertices are chosen.
\begin{lemma}\label{lem: char nonres pairs}
    In $P_m \square P_n$, suppose that the two boundary vertices on opposite sides are on the opposite horizontal sides at $v_{-1} = (p, n-1)$ and $v_0 = (q, 0)$ where $p, q \in \{0, 1, ..., m-1\}$ and $p < q$. Then, a pair of vertices is not resolved by $v_{-1}$ and $v_0$ if and only if it is of one of the following two forms:
    \begin{itemize}
        \item $\{(a,b), (a+t, b+t)\}$ where $t \in \mathbb{Z}, t >0$, $a, a+t \in [p, q]$ and $b, b+t \in [0, n-1]$,
        \item $\{(p - t, b), (q + t, b + q - p)\}$ where $t \in \mathbb{Z}, t >0$, $p-t, q+t \in [0, m-1]$, and $b, b+q - p\in [0, n-1]$.
    \end{itemize}
\end{lemma}
\begin{proof}
    Let $u=(x,y)$ and $w=(x',y')$ be distinct vertices with $x \leq x'$ in $V_{m,n}$ not resolved by $v_{-1}=(p,n-1)$ or $v_0=(q,0)$. If $x = x'$, then by either distance equation, $y = y'$, contradicting that $u$ and $w$ are distinct. Thus, $x < x'$. Then, we must have
    \begin{align*}
        d((x,y), (p, n-1)) = d((x', y'), (p, n-1)) \iff |x - p| + (n-1) - y = |x' - p| + (n-1) - y'\\
        \iff y' - y = |x' - p| - |x-p|,
        \\
        d((x,y), (q, 0)) = d((x', y'), (q, 0)) \iff |x-q| + y = |x' - q| + y'
        \iff y' - y = |x-q| - |x' - q|
        \\
        \implies
        |x' - p| - |x - p| = |x - q| - |x' - q| \iff |x' - p| + |x' - q| = |x-p| + |x-q|.
    \end{align*}
    Define $g(r) = |r - p| + |r - q|$. We find that $g(r)$ is constant for $r \in [p, q]$, strictly decreasing for $[0, p-1]$, and strictly increasing for $[q+1, m-1]$. Thus, for $x < x'$, $g(x) = g(x')$ implies that either $x, x' \in [p, q]$ or $x \in [0, p-1]$, $x' \in [q+1, m-1]$.

    If $x, x' \in [p, q]$, then $y' - y = (x' - p) - (x - p) = x' - x$, so we have that $\{(x,y), (x', y')\}$ is of the form $\{(a, b),(a+t, b+t)\}$ where $t \in \mathbb{Z}, t >0$, $a, a+t \in [p, q]$ and $b, b+t \in [0, n-1]$.

    If $x \in [0, p-1]$ and $x' \in [q+1, m-1]$, then the equality also implies $x = p - t$ and $x' = q + t$ for some $t \in \mathbb{Z}, t > 0$. Then, $y' - y = (x' - p) - (p - x) = -2p + x' + x = q-p$, so $\{(x,y), (x', y')\}$ is of the form $\{(p-t, b),(q+t, b+q-p)\}$ where $t \in \mathbb{Z}, t >0$, $p-t, q+t \in [0, m-1]$, and $b, b+q - p\in [0, n-1]$.

    Conversely, direct substitution shows that every pair of vertices of one of these two forms is unresolved by $\{(p, n-1), (q, 0)\}$. 
\end{proof}

Let $F_{=1}$ pairs be all pairs of vertices of the form $\{(a,b), (a+t, b+t)\}$ where $t \in \mathbb{Z}, t >0$, $a, a+t \in [p, q]$ and $b, b+t \in [0, n-1]$. Let $F_{<|1|}$ pairs be all pairs of vertices of the form $\{(p - t, b), (q + t, b + q - p)\}$ where $t \in \mathbb{Z}, t >0$, $p-t, q+t \in [0, m-1]$, and $b, b+q - p\in [0, n-1]$.

\begin{figure}[htbp]
  \centering
  \begin{subfigure}[t]{0.48\textwidth}
    \centering
    \begin{tikzpicture}[scale=0.59,font=\scriptsize]
      \fill[activeviolet!7] (1.80,-0.18) rectangle (4.20,5.18);
      \integergrid{6}{5}
      \draw[activeviolet,dashed,thick] (2,-0.18)--(2,5.18);
      \draw[activeviolet,dashed,thick] (4,-0.18)--(4,5.18);

      \node[initialvertex,label={[font=\scriptsize]above left:
        $v_{-1}=(p,n-1)$}] at (2,5) {};
      \draw[pairred,very thick] (2,1)--(4,3);
      \node[pairendpoint,label={[font=\scriptsize]below left:$(a,b)$}]
        at (2,1) {};
      \node[pairendpoint,label={[font=\scriptsize]above right:
        $(a+t,b+t)$}] at (4,3) {};
      \node[below=4pt] at (2,0) {$p$};
      \node[below=4pt] at (4,0) {$q$};
      \node[fill=white,inner sep=2pt] at (3,5.55)
        {$Q_0$-distance vector $(4,3)$ for both endpoints};
    \end{tikzpicture}
    \caption{An $F_{=1}$ pair: both $x$-coordinates lie in $[p,q]$.}
  \end{subfigure}\hfill
  \begin{subfigure}[t]{0.48\textwidth}
    \centering
    \begin{tikzpicture}[scale=0.59,font=\scriptsize]
      \fill[activeviolet!7] (1.80,-0.18) rectangle (4.20,5.18);
      \integergrid{6}{5}
      \draw[activeviolet,dashed,thick] (2,-0.18)--(2,5.18);
      \draw[activeviolet,dashed,thick] (4,-0.18)--(4,5.18);

      \node[initialvertex,label={[font=\scriptsize]above left:
        $v_{-1}=(p,n-1)$}] at (2,5) {};
      \draw[pairred,very thick] (1,1)--(5,3);
      \node[pairendpoint,label={[font=\scriptsize]below left:
        $(p-t,b)$}] at (1,1) {};
      \node[pairendpoint,label={[font=\scriptsize]above right:
        $(q+t,b+q-p)$}] at (5,3) {};
      \node[below=4pt] at (2,0) {$p$};
      \node[below=4pt] at (4,0) {$q$};
      \node[fill=white,inner sep=2pt] at (3,5.55)
        {$Q_0$-distance vector $(5,4)$ for both endpoints};
    \end{tikzpicture}
    \caption{An $F_{<|1|}$ pair: the endpoints lie outside the strip.}
  \end{subfigure}
  \caption{The two forms of vertex pairs left unresolved by $Q_0=\{(p,n-1),(q,0)\}$ in Lemma 3.1. The displayed distance vectors use $m=7$, $n=6$, $p=2$, $q=4$, and $t=2$ in (a), while (b) uses $t=1$.}
  \label{fig:pair-types}
\end{figure}
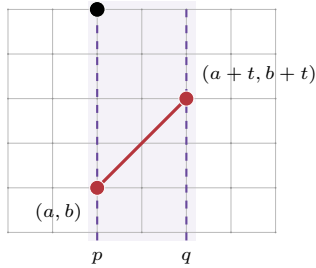
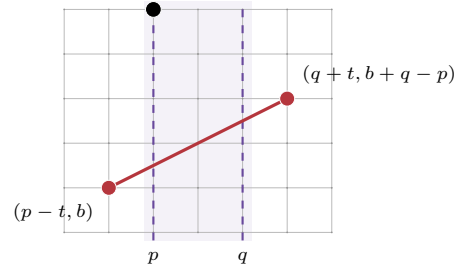

Therefore, to construct a minimal resolving set of $P_m \square P_n$, once two opposite boundary vertices have been chosen, the only pairs of unresolved vertices we must consider are the ones stated in Lemma \ref{lem: char nonres pairs}. 

For the following lemmas, given $\alpha_1, \alpha_2, x \in \mathbb{Z}$ where $\alpha_1 < \alpha_2$, define
\begin{equation*}
    \Delta_{\alpha_1, \alpha_2} (x) = |x - \alpha_1| - |x - \alpha_2| = \begin{cases}
        \alpha_1 - \alpha_2 &\text{if } x \leq \alpha_1 \\
        2x - \alpha_1 - \alpha_2 &\text{if } \alpha_1 < x < \alpha_2 \\
        \alpha_2 - \alpha_1 &\text{if } \alpha_2 \leq x.
    \end{cases}
\end{equation*}
A vertex $(x,y)$ fails to resolve a pair of vertices $\{(\alpha_1, \beta_1),(\alpha_2, \beta_2)\}$ where $\alpha_1 < \alpha_2$ and $\beta_1 < \beta_2$ if and only if
\begin{align*}
    d((x,y), (\alpha_1, \beta_1)) = d((x,y), (\alpha_2, \beta_2)) 
    \iff |x - \alpha_1| + |y - \beta_1| = |x - \alpha_2| + |y- \beta_2| \\
    \iff \left(|x - \alpha_1| - |x - \alpha_2|\right) + \left(|y - \beta_1| - |y - \beta_2|\right) = 0 \iff  \Delta_{\alpha_1, \alpha_2}(x) + \Delta_{\beta_1, \beta_2} (y) = 0.
\end{align*}

\begin{lemma}\label{lem: char res vertices slope 1}
    Given a pair of vertices of the $F_{=1}$ form $\{(a,b), (a+t, b+t)\}$ where $t \in \mathbb{Z}_{>0}$ and $(a,b), (a+t, b+t) \in V_{m,n}$, the set of vertices that resolve the pair is
    \begin{equation*}
        V_{m,n} \setminus (\mathcal{R}(0, a, b+t, n-1) \cup \{(a+i,b+t-i) \mid i \in \mathbb{Z}, 1 \leq i \leq t-1\} \cup \mathcal{R}(a+t, m-1, 0, b)),
    \end{equation*}
    Equivalently, all vertices that fail to resolve the pair are in the set
    \begin{equation*}
        \mathcal{R}(0, a, b+t, n-1) \cup \{(a+i,b+t-i) \mid i \in \mathbb{Z}, 1 \leq i \leq t-1\} \cup \mathcal{R}(a+t, m-1, 0, b).
    \end{equation*}
\end{lemma}
\begin{proof}
    We wish to find all $(x,y) \in V_{m,n}$ such that
    \begin{align*}
        d((x,y), (a,b)) = d((x,y), (a+t, b+t)) \iff \Delta_{a, a+t}(x) + \Delta_{b, b+t} (y) = 0.
    \end{align*}
    If $x \leq a$, then $\Delta_{a, a+t}(x) = -t$. Then, we must have $\Delta_{b, b+t}(y) = t \implies y \geq b + t$, so $(x,y) \in \mathcal{R}(0, a, b+t, n-1)$.

    If $a < x < a + t$, then $-t <\Delta_{a, a+t}(x) < t$. Then, we also need $-t < \Delta_{b, b+t}(y) < t$, so we must have $b < y < b+t$ and $2x -2a + 2y - 2b - 2t = 0 \implies x + y = a + b + t$. Thus, we must have
    \begin{equation*}
        (x,y) \in \{(a+i,b+t-i) \mid i \in \mathbb{Z}, 1 \leq i \leq t-1\}.
    \end{equation*}
    
    If $x \geq a + t$, then $\Delta_{a, a+t}(x) = t$, so $\Delta_{b,b+t}(y) = -t \implies y \leq b$, so $(x,y) \in \mathcal{R}(a+t, m-1, 0, b)$. 

    Therefore, for a vertex $(x,y) \in V_{m,n}$ to fail to resolve the pair of vertices $\{(a, b), (a+t, b+t)\}$, it must be in the set
    \begin{equation*}
        \mathcal{R}(0, a, b+t, n-1) \cup \{(a+i,b+t-i) \mid i \in \mathbb{Z}, 1 \leq i \leq t-1\} \cup \mathcal{R}(a+t, m-1, 0, b).
    \end{equation*}
    By substitution, we may check that all vertices in this set fail to resolve the pair of vertices $\{(a, b), (a+t, b+t)\}$, so the vertices that resolve the pair of vertices $\{(a,b), (a+t, b+t)\}$ are exactly the vertices in the set
    \begin{equation*}
        V_{m,n} \setminus (\mathcal{R}(0, a, b+t, n-1) \cup \{(a+i,b+t-i) \mid i \in \mathbb{Z}, 1 \leq i \leq t-1\} \cup \mathcal{R}(a+t, m-1, 0, b)).
    \end{equation*}
\end{proof}

\begin{lemma}\label{lem: char res vertices slope < |1|}
    Given a pair of vertices $\{(p-t, b), (q + t, b+ q - p)\}$ of the form $F_{<|1|}$ where $t \in \mathbb{Z}_{>0}$, $(p-t,b), (q+t, b+q-p) \in V_{m,n}$, and $q > p$, the set of all vertices that resolve the pair is the following set
    \begin{equation*}
        V_{m,n} \setminus \left(\mathcal{R}(p, p, b+q - p, n-1) \cup \{(p + i, b +q - p - i) \mid i \in \mathbb{Z},  1 \leq i \leq q - p-1 \} \cup \mathcal{R}(q, q, 0, b) \right).
    \end{equation*}
\end{lemma}
\begin{proof}
    We wish to find all $(x,y) \in V_{m,n}$ such that
    \begin{align*}
        d((x,y), (p-t,b)) = d((x,y), (q+t, b+q-p)) \iff \Delta_{p-t, q+t}(x) + \Delta_{b, b+q-p} (y) = 0.
    \end{align*}
    If $x \leq p -t$, then $\Delta_{p - t, q + t}(x) = p - q - 2t < p - q$. Since $\Delta_{b, b+q - p}(y) \leq q- p$, the equality $\Delta_{p-t, q+t}(x) + \Delta_{b, b+q-p} (y) = 0$ cannot hold.

    Likewise, if $x \geq q + t$, then $\Delta_{p -t, q + t}(x) = 2t + q - p > q - p$. Since $\Delta_{b, b+ q - p} (y) \geq p - q$, the equality $\Delta_{p - t, q + t} (x) + \Delta_{b, b+q - p} (y) = 0$ cannot hold either, so we omit both cases.

    If $p-t < x < q + t$, then $\Delta_{p-t, q + t}(x) = 2x - p -q$. Then, if $y \leq b$, we have $(2x - p - q) + (p-q) =2x - 2q = 0 \implies x = q$, so $(x,y) \in \mathcal{R}(q, q, 0, b)$.
    
    If $b < y < b + q - p$, then $(2x - p - q) + (2y - 2b +p - q) = 0 \implies x + y = b + q$, so $(x,y) \in \{(p + i, b +q - p - i) \mid i \in \mathbb{Z},  1 \leq i \leq q - p-1 \}$.
    
    If $y \geq b + q - p$, then $(2x - p - q) + (q - p) = 0 \implies x = p$, so $(x,y) \in \mathcal{R}(p, p, b+ q-p, n-1)$.
    
    We may check that all vertices in these three sets fail to resolve the pair of vertices $\{(p - t, b), (q+t, b + q - p)\}$, so the vertices that resolve the pair $\{(p - t, b), (q+t, b + q - p)\}$ are exactly the vertices in the set
    \begin{equation*}
        V_{m,n} \setminus (\mathcal{R}(p, p, b+q - p, n-1) \cup \{(p + i, b +q - p - i) \mid i \in \mathbb{Z},  1 \leq i \leq q - p-1 \} \cup \mathcal{R}(q, q, 0, b)).
    \end{equation*}
\end{proof}
 The following theorem follows from Lemmas \ref{lem: char res vertices slope 1} and \ref{lem: char res vertices slope < |1|} and allows us to eliminate the $F_{<|1|}$ pairs from further consideration. A closely related reduction is presented in Lemma $7$ of \cite{adarepstein2016alg}, in terms of joint diagonals. We present the following version for continuity of notation, as the same notation used in Lemmas \ref{lem: char res vertices slope 1} and \ref{lem: char res vertices slope < |1|} is used in Construction \ref{construction}.
 
\begin{theorem}\label{thm: eliminating F_{<|1|}}
    In the grid $P_m \square P_n$, given a set of vertices $S$ containing opposite-side boundary vertices $(q, 0)$ and $(p, n-1)$ where $0 \leq p < q \leq m-1$, if all pairs of vertices of the $F_{=1}$ form are resolved by $S$, then all pairs of vertices of the $F_{<|1|}$ form are also resolved by $S$. Consequently, $S$ is a resolving set of $P_m \square P_n$ if and only if it resolves all pairs of vertices of the $F_{=1}$ form determined by $p,q$ where $p<q$.
\end{theorem}
\begin{proof}
    We prove the contrapositive. Suppose there is some pair of vertices of the $F_{<|1|}$ form, $\{(p-t, b), (q + t, b + q - p)\}$, that is not resolved by $S$. Then, $S$ cannot contain any vertices in the following set:
    \begin{equation*}
        V_{m,n} \setminus \left(\mathcal{R}(p, p, b+q - p, n-1)\cup \{(p + i, b + q - p - i) \mid 0 < i < q - p \} \cup \mathcal{R}(q, q, 0, b)\right).
    \end{equation*}
    We find that the set of all vertices that resolve the pair of vertices $\{(p, b), (q, b+q - p)\}$
    of the $F_{=1}$ form,
    \begin{equation*}
        V_{m,n} \setminus \left(\mathcal{R}(0, p, b+q - p, n-1)\cup \{(p + i, b + q - p - i) \mid 0 < i < q - p \} \cup \mathcal{R}(q, m-1, 0, b)\right),
    \end{equation*}
    is a subset of the set of all vertices that resolve the pair $\{(p-t, b), (q + t, b + q - p)\}$.
    
    Thus, no vertex in $S$ resolves the pair of vertices $\{(p, b), (q, b+q - p)\}$ of the $F_{=1}$ form.

    We have demonstrated that if there is some unresolved pair of the $F_{<|1|}$ form, then there is some unresolved pair of the $F_{=1}$ form. Therefore, if there is no unresolved pair of the $F_{=1}$ form, then there is no unresolved pair of the $F_{<|1|}$. Therefore, if $S$ resolves all pairs of the form $F_{=1}$, and therefore all pairs of the form $F_{<|1|}$, $S$ is a resolving set of $P_m \square P_n$.
\end{proof}

Together, we have characterized the pairs of vertices that are unresolved by $(q, 0)$ and $(p, n-1)$ and have also characterized the vertices in $P_m \square P_n$ that do successfully resolve the aforementioned pairs.

\section{Construction Characterizing All Minimal Resolving Sets} \label{section: Construction}

In this section, we give a recursive construction generating all minimal resolving sets of $P_m \square P_n$ of size at least $4$. We state the construction in one fixed orientation, with opposite-side boundary vertices $(p, n-1)$ and $(q, 0)$ where $q > p$. The remaining orientations are obtained by applying symmetries of the grid. 

Adar and Epstein \cite{adarepstein2016alg} showed that every minimal resolving set of cardinality at least $4$ can be ordered to form a sequence corresponding to a zigzag sequence. Our construction refines this structural description and presents an exact criterion for minimality.

We first state the recursive construction in full. Its definition uses the active-rectangle invariant proved immediately afterwards, while its starting and terminal restrictions forced by minimality are justified later in the section. The proof is based on tracking the pairs that remain unresolved as the construction adds vertices to create a minimal resolving set. By Theorem \ref{thm: eliminating F_{<|1|}}, once $Q_0$ is fixed, it is sufficient to consider the pairs of type $F_{=1}$ that remain unresolved. We show that at every non-terminal stage $i$, all unresolved $F_{=1}$ pairs are parameterized by an active rectangle 
\begin{equation*}
    Z_i = \mathcal{Z}(x_i^-, x_i^+, y_i^-, y_i^+)
\end{equation*}
Two extremal pairs in this rectangle, which we denote as $\zeta_i^u$ and $\zeta_i^b$, determine two regions $R_i^u$ and $R_i^b$. Any minimal resolving set extending the current prefix of the construction by resolving $\zeta_i^u$ and $\zeta_i^b$ must contain a vertex from each of the regions. Choosing an admissible vertex from each region updates the bounds of the active rectangle and either terminates the construction or moves the construction into its next stage $i+1$. Thus, the construction alternates between shrinking its horizontal and vertical bounds until the active rectangle is empty and all pairs of type $F_{=1}$ and by Theorem \ref{thm: eliminating F_{<|1|}} therefore all pairs of vertices are resolved. 

We then consider initial and terminal restrictions specified in the construction through casework, as not every sequence of such region choices yields a minimal resolving set. Lemmas \ref{lem: cutting of R_0}, \ref{lem: perpendicular pair restrictions}, \ref{lem: domination within R}, and \ref{lem: forcing termination to hold} show that minimality forces additional boundary, initial, and terminal restrictions, as stated in Construction \ref{construction}. 

Finally, after establishing the active-rectangle invariant and termination, we prove the two directions of the characterization. For soundness, we show that every output of the construction is resolving and that every selected vertex is the unique resolver of some pair of vertices, so no vertex may be removed without destroying resolvability. For completeness, we begin with an arbitrary minimal resolving set and inductively prove that the unresolved vertex pairs force it to contain exactly one vertex in each region $R_i^u$ and $R_i^b$ at every stage $i$ of the construction. Therefore, every minimal resolving set may be recovered by the construction.

\begin{construction} \label{construction}
    Begin with grid $P_m \square P_n$ with vertices labeled $(x,y)$ for $x \in [0,m-1]$ and $y \in [0, n-1]$. Choose two opposite boundary vertices $v_{-1} = (p, n-1), v_0 = (q, 0)$ with $p < q$ on the horizontal sides of the grid. By Proposition $1$ of \cite{andersen2016minimum}, a minimal resolving set of cardinality at least $3$ cannot contain more than one corner vertex, so at most one of $v_{-1}$ and $v_0$ may be a corner vertex. Define $Q_0 = \{v_{-1}, v_0\}$.

    For every $i \geq 0$ after $Q_i$ has been selected, define 
    \begin{equation*}
        \hat{\mathcal{Q}}_i := Q_0 \cup \dots \cup Q_i.
    \end{equation*}
    For every $i \geq 0$, define $Z_i$ to be the set of unordered pairs of vertices of the $F_{=1}$ form that are not resolved by $\hat{\mathcal{Q}}_i$. If $Z_i = \varnothing$, the construction terminates and outputs $\hat{\mathcal{Q}}_i$.

    As proved in Lemma \ref{lem: induction}, whenever $Z_i \neq \varnothing$, there exist unique integers $x_i^{-} < x_i^{+}$ and $ y_i^- < y_i^+$ such that 
    \begin{equation*}
        Z_i = \mathcal{Z}(x_i^-, x_i^+, y_i^-, y_i^+).
    \end{equation*}
    We call $(x_i^-, x_i^+, y_i^-, y_i^+)$ the active bounds at stage $i$. Define $\zeta_i^u$ and $\zeta_i^b$ as specific elements of $Z_i$:
    \begin{align*}
        \zeta_i^u = \{(x_i^{-}, y_i^{+}-1),(x_i^{-}+1, y_i^{+})\}
        \qquad
        \zeta_i^b = \{(x_i^{+}-1, y_i^{-}),(x_i^{+}, y_i^{-}+1)\}.
    \end{align*}
    Then, define $R_i^u$ and $R_i^b$ as the sets of vertices that resolve $\zeta_i^u$ and $\zeta_i^b$ respectively without resolving $\zeta_h^u$ or $\zeta_h^b$ for any $0 \leq h \leq i-1$. That is,
    \begin{align*}
        R_i^u= \Res(\zeta_i^u) \setminus \bigcup_{0 \leq h < i} \left(\Res(\zeta_h^u) \cup \Res(\zeta_h^b) \right)
        \qquad
        R_i^b= \Res(\zeta_i^b) \setminus \bigcup_{0 \leq h < i} \left(\Res(\zeta_h^u) \cup \Res(\zeta_h^b) \right).
    \end{align*}
    Lemma \ref{lem: induction} also proves that $R_i^u$ and $R_i^b$ are nonempty and disjoint at every non-terminal stage where $i>0$.
    Then, define $Q_{i+1} = \{v_{2i+1}, v_{2i+2}\}$,
    where, if $i$ is odd, choose $v_{2i+1} \in R_i^u$ and $v_{2i+2} \in R_i^b$. If $i$ is even, choose $v_{2i+1} \in R_i^b$ and $v_{2i+2} \in R_i^u$. For all $v_{\iota}$, we define $v_{\iota} = (x_{\iota}, y_{\iota})$. The vertex choices are subject to the following additional restrictions. For every $v_{\iota}$ where $\iota \geq 1$, we must have $y_{\iota} \in [1, n-2]$. This is our boundary restriction.

    We have the following restrictions on $Q_1 = \{v_1, v_2\}$. If neither $v_{-1}$ nor $v_0$ is a corner, then if $x_1 = m-1$ and $x_2=0$, we must have $y_1 > y_2$. If $v_0 \neq (m-1,0)$ and $v_{-1} = (0, n-1)$, we must have $x_1 \neq m-1$, and if $v_{-1} \neq (0, n-1)$ and $v_0 = (m-1,0)$, we must have $x_2 \neq 0$. These are our $Q_1$ restrictions.

    Finally, we have the following forced-termination restrictions to preserve minimality. If $i=0$ and $q - p = 1$, we must have $y_1 \geq y_2$. By Lemma \ref{lem: induction}, this implies $Z_1 = \varnothing$, at which point the construction terminates and outputs $\hat{\mathcal{Q}}_1$. If $i \geq 2$ is even, $Z_i$ is nonempty, and $x_i^{+} - x_i^{-} = 1$, we must have $y_{2i+1} \geq y_{2i+2}$. By Lemma \ref{lem: induction}, this implies $Z_{i+1} = \varnothing$, at which point the construction terminates and outputs $\hat{\mathcal{Q}}_{i+1}$. If $i \geq 1$ is odd, $Z_i$ is nonempty, and $y_i^{+} - y_i^{-} = 1$, we must have $x_{2i+1} \geq x_{2i+2}$. By Lemma \ref{lem: induction}, this implies $Z_{i+1} = \varnothing$, at which point the construction terminates and outputs $\hat{\mathcal{Q}}_{i+1}$.
    
    If $Z_{i+1} \neq \varnothing$, repeat the construction with $i+1$. At each stage, the construction branches over all possible choices of $v_{2i+1}$ and $v_{2i+2}$ satisfying the conditions above.
\end{construction}

\begin{figure}[hbt!]
  \centering
  \begin{subfigure}[t]{0.48\textwidth}
    \centering
    \begin{tikzpicture}[scale=0.49,font=\scriptsize]
      \integergrid{8}{8}
      \draw[activebounds] (1.78,-0.22) rectangle (6.22,8.22);
      \node[initialvertex,label={[font=\scriptsize]above left:$v_{-1}$}]
        at (2,8) {};
      \node[initialvertex,label={[font=\scriptsize]below right:$v_0$}]
        at (6,0) {};

      \draw[lowerorange,thick] (2,7)--(3,8);
      \node[zetapoint] at (2,7) {};
      \node[zetapoint] at (3,8) {};
      \node[lowerorange,anchor=east] at (1.75,7.05) {$\zeta_0^u$};
      \draw[lowerorange,thick] (5,0)--(6,1);
      \node[zetapoint] at (5,0) {};
      \node[zetapoint] at (6,1) {};
      \node[lowerorange,anchor=west] at (6.25,0.9) {$\zeta_0^b$};
      \node[activeviolet,fill=white,inner sep=1.5pt] at (4,4)
        {$Z_0=Z(2,6,0,8)$};
    \end{tikzpicture}
    \caption{Choose $Q_0=\{(2,8),(6,0)\}$ and identify the initial
    active bounds and extremal pairs.}
  \end{subfigure}\hfill
  \begin{subfigure}[t]{0.48\textwidth}
    \centering
    \begin{tikzpicture}[scale=0.49,font=\scriptsize]
      % R_0^u = R(0,2,0,7) union R(3,8,8,8)
      \filldraw[upperregion, opacity = 0.5] (-0.20,-0.20) rectangle (2.20,7.20);
      \filldraw[upperregion, opacity = 0.5] (2.80,7.80) rectangle (8.20,8.20);
      % R_0^b = R(0,5,0,0) union R(6,8,1,8)
      \filldraw[lowerregion, opacity = 0.5] (-0.20,-0.20) rectangle (5.20,0.20);
      \filldraw[lowerregion, opacity = 0.5] (5.80,0.80) rectangle (8.20,8.20);
      \integergrid{8}{8}
      \node[initialvertex] at (2,8) {};
      \node[initialvertex] at (6,0) {};
      \node[firstpair,label={[font=\scriptsize]right:$v_1=(7,3)$}]
        at (7,3) {};
      \node[firstpair,label={[font=\scriptsize]left:$v_2=(1,5)$}]
        at (1,5) {};
      \draw[activebounds] (1.78,2.78) rectangle (6.22,5.22);
      \node[activeviolet,fill=white,inner sep=1.5pt] at (4,4)
        {$Z_1=Z(2,6,3,5)$};
      \node[upperblue,anchor=west,fill=white,inner sep=1pt]
        at (0.15,7.45) {$R_0^u$};
      \node[lowerorange,anchor=east,fill=white,inner sep=1pt]
        at (7.85,0.55) {$R_0^b$};
    \end{tikzpicture}
    \caption{Choose $v_1\in R_0^b$ and $v_2\in R_0^u$; their
    vertical coordinates become the new active bounds.}
  \end{subfigure}

  \vspace{1.2em}

  \begin{subfigure}[t]{0.48\textwidth}
    \centering
    \begin{tikzpicture}[scale=0.49,font=\scriptsize]
      % R_1^u = R(3,5,5,7), R_1^b = R(3,5,1,3)
      \filldraw[upperregion, opacity = 0.5] (2.80,4.80) rectangle (5.20,7.20);
      \filldraw[lowerregion, opacity = 0.5] (2.80,0.80) rectangle (5.20,3.20);
      \integergrid{8}{8}
      \node[initialvertex] at (2,8) {};
      \node[initialvertex] at (6,0) {};
      \node[firstpair] at (7,3) {};
      \node[firstpair] at (1,5) {};
      \node[secondpair,label={[font=\scriptsize]left:$v_3=(3,6)$}]
        at (3,6) {};
      \node[secondpair,label={[font=\scriptsize]right:$v_4=(5,2)$}]
        at (5,2) {};
      \draw[activebounds] (2.78,2.78) rectangle (5.22,5.22);
      \node[activeviolet,fill=white,inner sep=1.5pt] at (4,4)
        {$Z_2=Z(3,5,3,5)$};
      \node[upperblue,fill=white,inner sep=1pt] at (4,7.45) {$R_1^u$};
      \node[lowerorange,fill=white,inner sep=1pt] at (4,0.55) {$R_1^b$};
    \end{tikzpicture}
    \caption{Choose $v_3\in R_1^u$ and $v_4\in R_1^b$; their
    horizontal coordinates become the new active bounds.}
  \end{subfigure}\hfill
  \begin{subfigure}[t]{0.48\textwidth}
    \centering
    \begin{tikzpicture}[scale=0.49,font=\scriptsize]
      \filldraw[upperregion, opacity = 0.5] (2.75,3.75) rectangle (3.25,4.25);
      \filldraw[lowerregion, opacity = 0.5] (4.75,3.75) rectangle (5.25,4.25);
      \integergrid{8}{8}
      \node[initialvertex] at (2,8) {};
      \node[initialvertex] at (6,0) {};
      \node[firstpair] at (7,3) {};
      \node[firstpair] at (1,5) {};
      \node[secondpair] at (3,6) {};
      \node[secondpair] at (5,2) {};
      \draw[activebounds] (2.78,4) rectangle (5.22,4);
      \node[thirdpair,label={[font=\scriptsize]right:$v_5=(5,4)$}]
        at (5,4) {};
      \node[thirdpair,label={[font=\scriptsize]left:$v_6=(3,4)$}]
        at (3,4) {};
      \node[, activeviolet, fill=white,inner sep=1.5pt,align=center] at (4.15,5.10)
        {$Z_3=Z(3,5,4,4)=\varnothing$};
    \end{tikzpicture}
    \caption{The final pair has the opposite weak ordering, so the active
    set becomes empty and the construction terminates.}
  \end{subfigure}
  \caption{A complete admissible run of Construction 4.1 on $P_9 \square P_9$.  Black vertices form $Q_0$; green, blue, and red vertices form $Q_1,Q_2,Q_3$, respectively.  A dashed rectangle marks the coordinate bounds parameterizing $Z_i$, which is empty at the terminal stage. The $R_i^u$ and $R_i^b$ regions are shaded for each stage and are undefined at the terminal stage.} 
  \label{fig:construction-run}
\end{figure}
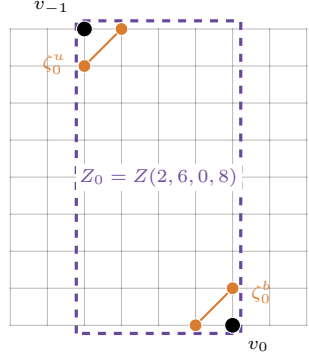
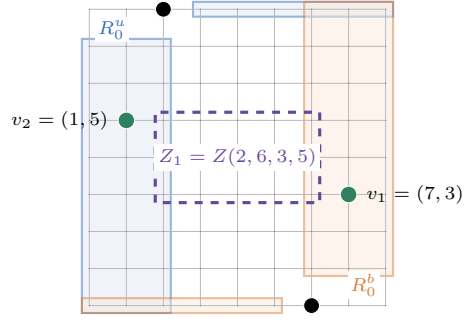
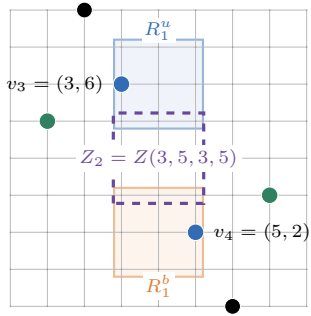
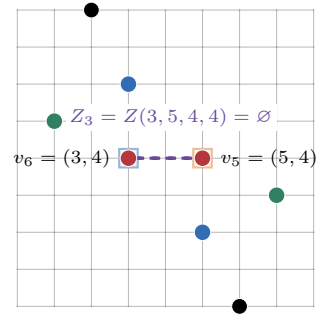

The following lemma concerns how $Z_i$ changes as pairs of vertices are added to $\hat{\mathcal{Q}}_i$.
\begin{lemma} \label{lem: active-rectangle-update}
    Suppose that for some $i \geq 0$, we have that $Z_i = \mathcal{Z}(x^-, x^+, y^-, y^+)$, where $x^-$, $x^+$, $y^-$, and $y^+$ are integers where $0 \leq x^- < x^+ \leq m-1$ and $0 \leq y^- < y^+ \leq n-1$. Then, suppose that the vertex $v$ is added to the set of vertices $\hat{\mathcal{Q}}_i$. Let $v = (x_v, y_v) \in V_{m,n}$. Let $Z_{i}(v)$ be the set of all pairs of vertices of the $F_{=1}$ form unresolved by $\hat{\mathcal{Q}}_i \cup \{v\}$. In each of the following cases, the new bounds for the unresolved pairs of vertices are as stated.
    
    If $x_v \in (x^-, x^+)$ and $y_v \geq y^+$, the $x$-bounds for $Z_{i}(v)$ become $[x_v, x^+]$ and the $y$-bounds do not change.

    If $x_v \in (x^-, x^+)$ and $y_v \leq y^-$, the $x$-bounds for $Z_{i}(v)$ become $[x^-, x_v]$ and the $y$-bounds do not change.

    If $x_v \leq x^-$ and $y_v \in (y^-, y^+)$, the $x$-bounds for $Z_{i}(v)$ do not change and the $y$-bounds become $[y^-, y_v]$.

    If $x_v \geq x^+$ and $y_v \in (y^-, y^+)$, the $x$-bounds for $Z_{i}(v)$ do not change and the $y$-bounds become $[y_v,y^+]$.

    If $x_v\leq x^-$ and $y_v\leq y^-$, or if $x_v\geq x^+$ and $y_v\geq y^+$, then $Z_{i}(v)=\varnothing$.
\end{lemma}
\begin{proof}
    Suppose that for some $i \geq 0$, we have $Z_i = \mathcal{Z}(x^-, x^+, y^-, y^+)$ where $x^-$, $x^+$, $y^-$, and $y^+$ are integers where $x^+>x^-$ and $y^+ > y^-$. Then, suppose that the vertex $v$ is added to the set of vertices $\hat{\mathcal{Q}}_i$.

    By Lemma \ref{lem: char res vertices slope 1}, the set of all pairs of vertices of the $F_{=1}$ form that $v$ does not resolve is
    \begin{align*}
        \mathcal{Z}(0, x_v, y_v, n-1) \cup \mathcal{Z}(x_v, m-1, 0, y_v) \cup \\\left\{\{(x_v-\alpha,y_v-\beta),(x_v+\beta,y_v+\alpha)\} \mid \alpha, \beta \in \mathbb{Z}_{>0}; x_v-\alpha,x_v+\beta \in[0,m-1];y_v-\beta,y_v+\alpha \in [0,n-1] \right\}.
    \end{align*}
    $Z_i$ is the set of all pairs of vertices of the $F_{=1}$ form that $\hat{\mathcal{Q}}_i$ does not resolve, so the set of all pairs of vertices of the $F_{=1}$ form that $\hat{\mathcal{Q}}_i \cup \{v\}$ does not resolve, $Z_i(v)$, is the intersection
    \begin{align*}
        \mathcal{Z}(x^-, x^+, y^-, y^+) \cap ( \mathcal{Z}(0, x_v, y_v, n-1) \cup \mathcal{Z}(x_v, m-1, 0, y_v) \\ 
        \cup \{ \{(x_v-\alpha,y_v-\beta),(x_v+\beta,y_v+\alpha)\} \mid \alpha, \beta \in \mathbb{Z}_{>0}; x_v-\alpha,x_v+ \beta \in[0,m-1];y_v-\beta,y_v+\alpha \in [0,n-1] \} )
    \end{align*}
    If $x_v \in (x^-, x^+)$ and $y_v \geq y^+$, then $b+t > y_v \geq y^+$ for every pair of vertices in the first set in the union, so this set is disjoint from $Z_i$ and may be discarded. Likewise, $y_v+\alpha > y^+$ for every pair of vertices in the third set in the union, so this set is disjoint from $Z_i$ and may be discarded. Therefore, this intersection is equal to $\mathcal{Z}(x_v,x^+, y^-, y^+)$, as desired.

    The other three non-terminal cases for $x_v$ and $y_v$ are proved similarly using Lemma \ref{lem: char res vertices slope 1} by taking the intersection between the set of pairs of vertices unresolved by $v$ and $Z_{i}$. The third set is always disjoint from $Z_i$, as the following inequalities hold:
    \begin{table}[hbt!]
    \centering
    \label{tab:my_table}
        \begin{tabular}{|l|c|}
            \hline
            Location of $v$ & Endpoint forced out of $Z_i$ bounds  \\ \hline
            $y_v \leq y^-$ & $y_v-\beta < y^-$ \\ \hline
            $x_v \leq x^-$ & $x_v-\alpha< x^-$ \\ \hline
            $x_v \geq x^+$ & $x_v+\beta > x^+$ \\ \hline
        \end{tabular}
    \end{table}

    Thus, we have
    \begin{align*}
        x_v \in (x^-, x^+), y_v \leq y^- \implies Z_i(v) = \mathcal{Z}(x^-, x_v, y^-, y^+)\\
        x_v \leq x^-, y_v \in (y^-, y^+) \implies Z_i(v) = \mathcal{Z}(x^-, x^+, y^-, y_v)\\
        x_v \geq x^+, y_v \in (y^-, y^+) \implies Z_i(v) = \mathcal{Z}(x^-, x^+, y_v, y^+).
    \end{align*}
    
    It remains to consider the terminal case. Suppose first that $x_v\leq x^-$ and $y_v\leq y^-$. For every pair $\{(a,b),(a+t,b+t)\}\in Z_i$, we have $x_v\leq x^-\leq a<a+t$ and $y_v\leq y^-\leq b<b+t$. Thus,
    \begin{equation*}
        d\bigl(v,(a+t,b+t)\bigr)
        =(a+t-x_v)+(b+t-y_v)\
        =(a-x_v)+(b-y_v)+2t\
        =d\bigl(v,(a,b)\bigr)+2t.
    \end{equation*}
    Since $t>0$, the vertex $v$ resolves every pair in $Z_i$. Therefore, $Z_{i}(v)=\varnothing$. The second terminal case is treated similarly, as we find $d(v, (a,b)) = d(v, (a+t, b+t)) + 2t$, which implies that $v$ resolves every pair in $Z_i$ and so $Z_i (v) = \varnothing$.
\end{proof}

For conciseness, define
\begin{equation*}
    U_i := \bigcup_{0 \leq h < i}(\Res(\zeta_h^u) \cup \Res(\zeta_h^b)), \qquad U_0 = \varnothing.
\end{equation*}

\begin{lemma} \label{lem: induction}
    Let $Q_0, Q_1, ..., Q_i$ with $Q_0 = \{v_{-1}, v_0\} = \{(p, n-1),(q, 0)\}$ be a labeled sequence satisfying the recursive region-selection and boundary rules of Construction \ref{construction} through stage $i$, without necessarily satisfying its $Q_1$ or terminal restrictions. Then, if stage $i$ is non-terminal, $Z_i = \mathcal{Z}(x_i^-, x_i^+, y_i^-, y_i^+)$, with the initial bounds $(x_0^-, x_0^+, y_0^-, y_0^+) =(p, q, 0, n-1)$. 
    
    If $i$ is odd, then $Z_{i+1} = \mathcal{Z}(x_{2i+1}, x_{2i+2}, y_i^-, y_i^+)$, so when $Z_{i+1} \neq \varnothing$, we have $(x_{i+1}^-, x_{i+1}^+, y_{i+1}^-, y_{i+1}^+) = (x_{2i+1}, x_{2i+2}, y_i^-, y_i^+)$. 
    
    If $i$ is even, then $Z_{i+1} = \mathcal{Z}(x_i^-, x_i^+, y_{2i+1}, y_{2i+2})$, so when $Z_{i+1} \neq \varnothing$, we have $(x_{i+1}^-, x_{i+1}^+, y_{i+1}^-, y_{i+1}^+) = (x_i^-, x_i^+, y_{2i+1}, y_{2i+2})$.

    The initial regions are
    \begin{equation*}
        R_0^u = \mathcal{R}(0, p, 0, n-2) \cup \mathcal{R}(p+1, m-1, n-1, n-1),
        \qquad
        R_0^b = \mathcal{R}(0, q-1, 0, 0) \cup \mathcal{R}(q, m-1, 1, n-1).
    \end{equation*}
    For every odd $i \geq 1$ with $Z_i \neq \varnothing$,
    \begin{equation*}
        R_i^u = \mathcal{R}(x_i^- + 1, x_i^+ -1, y_i^+, y_{i-1}^+ -1), \qquad R_i^b = \mathcal{R}(x_i^- + 1, x_i^+ -1, y_{i-1}^- + 1, y_i^-).
    \end{equation*}
    For every even $i \geq 2$ with $Z_i \neq \varnothing$,
    \begin{equation*}
        R_i^u = \mathcal{R}(x_{i-1}^- + 1, x_i^-, y_i^- + 1, y_i^+ - 1), \qquad R_i^b = \mathcal{R}(x_i^+, x_{i-1}^+ - 1, y_i^- + 1, y_i^+ -1).
    \end{equation*}
    Finally, for all $i$, we have that
    \begin{equation*}
        \bigcup_{0 \leq h < i} (R_h^u \cup R_h^b) = U_i.
    \end{equation*}
    If, in addition, the sequence is a prefix admitted by Construction~\ref{construction}, then at every non-terminal stage $R_i^u$ and $R_i^b$ are nonempty. They are disjoint for $i>0$, and one may choose one vertex from each region satisfying every applicable restriction of the construction.
\end{lemma}
\begin{proof}
    By definition, $R_i^\sigma = \Res(\zeta_i^\sigma)\setminus U_i$ for $\sigma \in \{u, b\}$. Thus, we have
    \begin{equation*}
        U_i\cup R_i^u\cup R_i^b = U_i \cup\Res(\zeta_i^u) \cup\Res(\zeta_i^b) = U_{i+1}.
    \end{equation*}
    By induction, we therefore have
    \begin{equation*}
        U_i = \bigcup_{0\leq h<i}(R_h^u\cup R_h^b),
    \end{equation*}
    as desired.

    $Z_0$ is the set of unordered pairs of vertices of the $F_{=1}$ form that are not resolved by $Q_0 = \{v_{-1}, v_0\}$. Since $v_{-1} = (p, n-1)$ and $v_0 = (q,0)$, by Lemma \ref{lem: char nonres pairs}, we have $Z_0 = \mathcal{Z}(p, q, 0, n-1)$, as desired. Thus, by the construction, we have that
    \begin{align*}
        \zeta_0^u = \{(p,n-2),(p+1,n-1)\} \qquad \zeta_0^b = \{(q-1,0),(q,1)\}.
    \end{align*}
    Then, $R_0^u$ and $R_0^b$ are the sets of vertices in $P_m \square P_n$ that resolve $\zeta_0^u$ and $\zeta_0^b$ respectively, so
    \begin{equation*}
        R_0^u \cup R_0^b = \Res(\zeta_0^u) \cup \Res(\zeta_0^b) = U_1,
    \end{equation*}
    as desired. By Lemma \ref{lem: char res vertices slope 1}, the set of all vertices that resolve $\zeta_0^u$ and the set of all vertices that resolve $\zeta_0^b$ are
    \begin{align*}
        \mathcal{R}(0, p, 0, n-2) \cup \mathcal{R}(p+1, m-1, n-1, n-1), \qquad \mathcal{R}(0, q-1, 0,0) \cup \mathcal{R}(q, m-1, 1, n-1),
    \end{align*}
    respectively, as desired. 
    
    Now, for the second base case, applying Lemma \ref{lem: active-rectangle-update} to $v_1$ and $v_2$, we find that $Z_1 = \mathcal{Z}(p, q, y_1, y_2)$. If $y_1 \geq y_2$, then $Z_1 = \varnothing$ and our construction terminates. Thus, assume $y_1 < y_2$.
    
    Therefore, as defined by Construction \ref{construction}, we have
    \begin{align*}
        \zeta_1^u = \{(p,y_{2}-1),(p+1,y_{2})\} \qquad \zeta_1^b = \{(q-1,y_1),(q,y_1+1)\}.
    \end{align*}
    Then, $R_1^u$ and $R_1^b$ are the sets of vertices that resolve $\zeta_1^u = \{(p,y_2-1),(p+1, y_2) \}$ and $\zeta_1^b = \{(q-1, y_1),(q, y_1+1)\}$ respectively without resolving either $\zeta_0^u$ or $\zeta_0^b$. By Lemma \ref{lem: char res vertices slope 1}, the set of vertices that resolve $\zeta_1^u$ is
    \begin{align*}
        V_{m,n} \setminus (\mathcal{R}(0, p, y_2, n-1) \cup \mathcal{R}(p+1, m-1, 0, y_2 - 1)) = \mathcal{R}(0, p, 0, y_2 - 1) \cup \mathcal{R}(p+1, m-1, y_2, n-1).
    \end{align*}
    To obtain $R_1^u$, we subtract $R_0^u$ and $R_0^b$. Since $v_2 \in R_0^u$, we have that $0 \leq x_2 \leq p$ and $1 \leq y_2 \leq n-2$, so we have $R_1^u = \mathcal{R}(p+1, q-1, y_2, n-2)$, as desired. Likewise, by Lemma \ref{lem: char res vertices slope 1}, the set of vertices that resolve $\zeta_1^b$ is
    \begin{align*}
        V_{m,n} \setminus (\mathcal{R}(0, q-1, y_1 + 1, n-1) \cup \mathcal{R}(q, m-1, 0, y_1)) = \mathcal{R}(0, q-1, 0, y_1) \cup \mathcal{R}(q, m-1, y_1 + 1, n-1).
    \end{align*}
    To obtain $R_1^b$, we subtract $R_0^u$ and $R_0^b$. Since $v_1 \in R_0^b$, we have that $q \leq x_1 \leq m-1$ and $1 \leq y_1 \leq n-2$, so we have $R_1^b = \mathcal{R}(p + 1, q-1, 1, y_1)$, as desired.
        
    Now, suppose all assertions hold at some non-terminal stage $i \geq 1$. If $i$ is odd, we have $Q_{i+1} = \{v_{2i+1}, v_{2i+2}\}$ with $v_{2i+1} \in R_i^u$ and $v_{2i+2} \in R_i^b$.
    
    Then, consider $Q_{i+1} = \{v_{2i+1}, v_{2i+2}\}$. Since $v_{2i+1} \in R_i^u$ and $v_{2i+2} \in R_i^b$, we have that
    \begin{align*}
        x_{2i+1} \in [x_i^- + 1, x_i^+ -1], \qquad
        y_{2i+1} \in [y_i^+, y_{i-1}^+ -1],
        \qquad
        x_{2i+2} \in [x_i^- +1, x_i^+ -1],
        \qquad
        y_{2i+2} \in [y_{i-1}^- +1, y_i^-].
    \end{align*}
    Applying Lemma \ref{lem: active-rectangle-update} first to $v_{2i+1}$ and then to $v_{2i+2}$, we obtain $Z_{i+1} = \mathcal{Z}(x_{2i+1}, x_{2i+2}, y_i^-, y_i^+)$, as desired.

    If $x_{2i+1} \geq x_{2i+2}$, then the second vertex falls into a terminal case of Lemma \ref{lem: active-rectangle-update}, and both sides of the equation are empty. Hence, for the remainder of the calculation at stage $i+1$, suppose that $Z_{i+1}\neq\varnothing$. Equivalently, $x_{2i+1}<x_{2i+2}$, and the two vertices replace both horizontal bounds. The distinguished pairs and regions at stage $i+1$ are therefore defined.

    Then, as defined by Construction \ref{construction}, $R_{i+1}^u$ and $R_{i+1}^b$ are the sets of vertices that resolve the distinguished pairs
    \begin{align*}
        \zeta_{i+1}^u = \{(x_{2i+1}, y_i^+ -1),(x_{2i+1}+1, y_i^+) \},
        \qquad
        \zeta_{i+1}^b =\{(x_{2i+2}-1, y_i^-),(x_{2i+2}, y_i^- +1)\}
    \end{align*}
    respectively without resolving any of $\zeta_h^u, \zeta_h^b$ for $0 \leq h \leq i$. Thus, we have
    \begin{align*}
        R_{i+1}^u = \Res(\zeta_{i+1}^u) \setminus U_{i+1}
        = \Res(\zeta_{i+1}^u) \setminus \bigcup_{0 \leq h < i+1}(R_h^u \cup R_h^b)
        \\
        R_{i+1}^b = \Res(\zeta_{i+1}^b) \setminus U_{i+1}
        = \Res(\zeta_{i+1}^b) \setminus \bigcup_{0 \leq h < i+1}(R_h^u \cup R_h^b).
    \end{align*}
    We find from Lemma \ref{lem: char res vertices slope 1} and from our induction hypothesis that
    \begin{align*}
        R_{i+1}^u= \left(\mathcal{R}(0, x_{i+1}^-, 0, y_{i+1}^+ - 1) \cup \mathcal{R}(x_{i+1}^- + 1, m-1, y_{i+1}^+, n-1) \right)\setminus \bigcup_{0 \leq h < i+1} (R_h^u \cup R_h^b)\\
        = 
        \left( \mathcal{R}(0, x_{i+1}^-, 0, y_{i+1}^+ - 1) \cup \mathcal{R}(x_{i+1}^- + 1, m-1, y_{i+1}^+, n-1) \right) \setminus \left( U_i \cup \Res(\zeta_i^u) \cup \Res(\zeta_i^b)\right)\\
        =
        \mathcal{R}(x_i^- +1, x_{2i+1}, y_i^- +1, y_i^+ -1) \setminus U_i = \mathcal{R}(x_i^- +1, x_{2i+1}, y_i^- +1, y_i^+ -1),
    \end{align*}
    as by our induction hypothesis and the explicit formulas for $R_h^u$ and $R_h^b$, since all vertices in $R_h^u$ and $R_h^b$ with $h < i$ lie outside the interior of the region $\mathcal{R}(x_i^- + 1, x_i^+ -1, y_i^- + 1, y_i^+ -1)$, $U_i$ and $\mathcal{R}(x_i^- +1, x_{2i+1}, y_i^- +1, y_i^+ -1)$ are disjoint. Likewise, we may find
    \begin{align*}
        R_{i+1}^b = \mathcal{R}(x_{2i+2}, x_i^+ -1,
        y_i^ - +1, y_i^+ -1).
    \end{align*}
    These are the asserted even-stage formulas, as desired. 

    If $i$ is even, we apply the same arguments using Lemma \ref{lem: active-rectangle-update} to find that $Z_{i+1} = \mathcal{Z}(x_i^-, x_i^+, y_{2i+1}, y_{2i+2})$. If $Z_{i+1} = \varnothing$, then our construction terminates, so suppose $Z_{i+1} \neq \varnothing \iff y_{2i+1} < y_{2i+2}$. Likewise, use Lemma \ref{lem: char res vertices slope 1} to find 
    \begin{equation*}
        R_{i+1}^u = \mathcal{R}(x_i^- + 1, x_i^+ - 1, y_{2i+2}, y_i^+ - 1), \qquad R_{i+1}^b = \mathcal{R}(x_i^- + 1, x_i^+ - 1, y_i^- + 1, y_{2i+1}).
    \end{equation*}
    These are the asserted odd-stage formulas, as desired.

    It remains to prove that at every stage $i$ for which $Z_i \neq \varnothing$, $R_i^u$ and $R_i^b$ are nonempty, and when $i > 0$, they are disjoint. 

    For $i = 0$, after imposing the restriction $y_1, y_2 \in [1, n-2]$ on $R_0^u$ and $R_0^b$, we may choose $v_1 = (q, 1)$ and $v_2 = (p, 1)$, so the admissible portions of $R_0^u$ and $R_0^b$ are nonempty. 

    Let $i \geq 1$ and suppose $Z_i \neq \varnothing$. If $i$ is odd, then the termination restrictions force $x_i^+ - x_i^- \geq 2$, as otherwise the termination restrictions would have forced $Z_i = \varnothing$. Therefore, $x_i^- + 1 \leq x_i^+ -1$.

    The vertices chosen at the $i-1$ stage are in the vertical interval $[y_{i-1}^- + 1, y_{i-1}^+ - 1]$, so $y_{i-1}^- + 1 \leq y_i^- < y_i^+ \leq y_{i-1}^+ - 1$. This implies both $R_i^u$ and $R_i^b$ are nonempty. They are disjoint, as $R_i^u$ has vertical interval $[y_i^+, y_{i-1}^+ -1]$, $R_i^b$ has vertical interval $[y_{i-1}^- + 1, y_i^-],$ and $y_i^ - < y_i^+$.

    If $i$ is even, then the termination restrictions force $y_i^+ - y_i^ - \geq 2 \implies y_i^- + 1 \leq y_i^+ -1$. The vertices chosen at the $i-1$ stage satisfy $x_{i-1}^-+1\leq x_i^-<x_i^+ \leq x_{i-1}^+-1$, implying both $R_i^u$ and $R_i^b$ are nonempty. They are also disjoint because their respective horizontal intervals are disjoint.

    Finally, suppose a terminal restriction applies. If $i$ is odd and $y_i^+-y_i^-=1$, choose the two new vertices with the same horizontal coordinate in the common interval $[x_i^-+1,x_i^+-1]$. Then $x_{2i+1}=x_{2i+2}$, so the required inequality holds. 
    
    If $i$ is even and $x_i^+-x_i^-=1$, choose the two vertices with the same vertical coordinate in $[y_i^-+1,y_i^+-1]$. Then $y_{2i+1}=y_{2i+2}$. In either case, the selected vertices are distinct because $R_i^u$ and $R_i^b$ are disjoint.

    Therefore an admissible next pair exists at every non-terminal stage. If $\mathcal{Z}(A,B,C,D)\neq\varnothing$, then $A$ and $B$ are respectively the minimum first $x$-coordinate and maximum second $x$-coordinate among its canonically $x$-ordered pairs, and $C,D$ are the analogous extrema of the $y$-coordinates. Hence $A,B,C,D$ are uniquely determined by $\mathcal{Z}(A,B,C,D)$.
\end{proof}

\begin{corollary} \label{coro: induction result of Z_i form}
    By Lemma \ref{lem: induction}, for any run of Construction \ref{construction} and every reached stage $i\geq1$,
    \begin{equation*}
        Z_i=
        \begin{cases}
        \mathcal Z(x_{2i-3},x_{2i-2},y_{2i-1},y_{2i}),
        & i\text{ odd},\\
        \mathcal Z(x_{2i-1},x_{2i},y_{2i-3},y_{2i-2}),
        & i\text{ even}.
        \end{cases}
    \end{equation*}
\end{corollary}

The following theorem ensures that Construction \ref{construction} always terminates.

\begin{theorem} \label{thm: termination}
    For any grid $P_m \square P_n$ with $m, n \geq 3$, Construction \ref{construction} terminates.
\end{theorem}
\begin{proof}
    By Corollary \ref{coro: induction result of Z_i form}, for $i \geq 1$, if $i$ is odd, then $Z_i = \mathcal{Z}(x_{2i-3},x_{2i-2}, y_{2i-1},y_{2i})$, and if $i$ is even, then $Z_i = \mathcal{Z}(x_{2i-1},x_{2i}, y_{2i-3},y_{2i-2})$.
    For every $i \geq 1$, define
    \begin{equation*}
        d_x^i = \begin{cases}
            x_{2i-2} - x_{2i-3} & \text{if } i \text{ is odd}
            \\
            x_{2i} - x_{2i-1} & \text{if } i \text{ is even}
        \end{cases}
        \qquad \qquad
        d_y^i = \begin{cases}
            y_{2i} - y_{2i-1} & \text{if } i \text{ is odd}
            \\
            y_{2i-2} - y_{2i-3} & \text{if } i \text{ is even.}
        \end{cases}
    \end{equation*}
    The construction terminates when $Z_i = \varnothing$, which happens when either $d_x^i \leq 0$ or $d_y^i \leq 0$. We may prove this must happen for some $i$ by showing that for $i \geq 1$, either $d_x^i > d_x^{i+1}$ or $d_y^i > d_y^{i+1}$.

    We have that for all odd $i \geq 1$, $Z_i = \mathcal{Z}(x_{2i-3},x_{2i-2}, y_{2i-1},y_{2i})$ and $Z_{i+1} = \mathcal{Z}(x_{2i+1},x_{2i+2}, y_{2i-1},y_{2i})$. We have that $d_y^i = d_y^{i+1}$. However, we have that
    \begin{align*}
        d_x^i = x_{2i-2} - x_{2i-3}
        \\
        v_{2i+2} \in R_i^b \implies x_{2i+2} \in [x_{2i-3} +1, x_{2i-2}-1] \implies x_{2i+2} < x_{2i-2}
        \\
        v_{2i+1} \in R_i^u \implies x_{2i+1} \in [x_{2i-3} +1, x_{2i-2}-1] \implies x_{2i+1} > x_{2i-3}
        \\
        \implies x_{2i+2} - x_{2i+1} < x_{2i-2} - x_{2i-3}
        \implies d_x^{i+1} < d_x^i.
    \end{align*}

    For all even $i \geq 2$, we have $Z_i = \mathcal{Z}(x_{2i-1},x_{2i}, y_{2i-3},y_{2i-2})$ and $Z_{i+1} = \mathcal{Z}(x_{2i-1},x_{2i}, y_{2i+1}, y_{2i+2})$. We have that $d_x^i = d_x^{i+1}$. However, we have that
    \begin{align*}
        d_y^i = y_{2i-2} - y_{2i-3}
        \\
        v_{2i+2} \in R_i^u \implies y_{2i+2} \in [y_{2i-3}+1,y_{2i-2}-1] \implies y_{2i+2} < y_{2i-2}
        \\
        v_{2i+1} \in R_i^b \implies y_{2i+1} \in [y_{2i-3}+1,y_{2i-2}-1] \implies y_{2i+1} > y_{2i-3}
        \implies d_y^{i+1} < d_y^i.
    \end{align*}
    If the run were infinite, $d_x^i$ would decrease by a positive integer amount at every odd transition and $d_y^i$ at every even transition. Since both are nonnegative integers, this is 
    impossible. Therefore, one of $d_x^i$, $d_y^i$ must eventually be nonpositive, at which point $Z_i = \varnothing$ and the construction terminates.
\end{proof}

The following lemmas characterize the pairs of $Q_{i+1}$ that must be discarded in order to ensure the minimality of the resolving set constructed by our construction.

\begin{lemma} \label{lem: cutting of R_0}
    Let $M$ be a minimal resolving set of cardinality at least $4$ containing $Q_0 = \{(p, n-1), (q, 0)\}$ with $q > p$. Then, $M \setminus Q_0$ contains no vertex on the horizontal boundaries of the grid. That is, every vertex in $M \setminus Q_0$ has $y$-coordinate in the interval $[1, n-2]$.
\end{lemma}
\begin{proof}
    Suppose without loss of generality that some vertex $w \in M \setminus Q_0$ is chosen on the horizontal boundary of the grid $y=0$. Given that $v_{-1} = (p, n-1)$ and $v_0 = (q, 0)$, we consider the following cases.

    If $x_{w} < p$, then the subset of vertices $\{v_{-1}, v_0, w\}$ either is a minimal resolving set by Theorem $2$ in \cite{andersen2016minimum}, or contains a pair of two corner vertices sharing a side, which is a minimal resolving set by \cite{melter1984metric}. In either case, it contains a proper resolving set of $M$, which breaks minimality of $M$, so we eliminate this case.

    If $p < x_{w} < q$, we find that $M \setminus \{v_0\}$ is a resolving set. Without $v_0$, we have opposite boundary vertices $v_{-1}$ and $w$, and by Lemma \ref{lem: char nonres pairs}, the vertex pairs of the $F_{=1}$ form that may be unresolved by this pair in $M$ are $\mathcal{Z}(p, x_w, 0, n-1)$.
    
    Since $x_w < q$, $v_0$ does not resolve any of these pairs of vertices, so all $F_{=1}$ vertex pairs must be resolved by $M \setminus \{v_0\}$. Therefore by Theorem \ref{thm: eliminating F_{<|1|}}, $M \setminus \{v_0\}$ must be a resolving set, breaking the minimality of $M$.

    Finally, if $q < x_{w}$, then we may find that $M \setminus \{w\}$ is a resolving set. Without $w$, we have opposite boundary vertices $v_{-1}$ and $v_0$, and by Lemma \ref{lem: char nonres pairs}, the vertex pairs of the $F_{=1}$ form that may be unresolved by this pair are $\mathcal{Z}(p, q, 0, n-1)$.
    
    Since $q < x_{w}$, $w$ does not resolve any of these pairs of vertices, so all such vertex pairs must be resolved by $M \setminus \{w\}$. Therefore, by Theorem \ref{thm: eliminating F_{<|1|}}, $M \setminus \{w\}$ must be a resolving set, breaking minimality of $M$.

    If $x_{w} = p$, then the set of vertices $\{v_{-1}, v_0, w\}$ either is a resolving set by Theorem $2$ in \cite{andersen2016minimum} or contains a pair of two corner vertices sharing a side, which is a minimal resolving set by \cite{melter1984metric}. In either case, it contains a proper resolving set of $M$, contradicting minimality, so we do not consider this case.

    If $x_{w} = q$, then $w = (q,0) = v_0$, so this case is not possible.

    We may treat the other boundary case similarly by applying the grid isomorphism $T : P_m \square P_n \to P_m \square P_n$, $T(x,y) = (m-1-x, n-1-y)$. Therefore, since $v_{-1} = (p, n-1)$ and $v_{0} = (q, 0)$, there can be no other vertex in our set with a $y$-coordinate of $0$ or $n-1$.
\end{proof}

\begin{lemma} \label{lem: perpendicular pair restrictions}
    Let $Q_0, \dots, Q_i$ satisfy the recursive vertex-selection rules of Construction \ref{construction}, except possibly the $Q_1$ and termination restrictions, and suppose $i \geq 1$ and $\hat{\mathcal{Q}}_i$ resolves $P_m \square P_n$. Suppose that $x_1= m - 1$ and $x_2 = 0$. Then,
    \begin{itemize}
        \item If $y_1 < y_2$, the set $Q_1 \cup \dots \cup Q_i$ is resolving, so $\hat{\mathcal{Q}}_i$ is not minimal.
        \item If $y_1 = y_2$, $\hat{\mathcal{Q}}_i$ contains a $3$-minimal subset and is not minimal.
        \item If $y_1 > y_2$, then $Z_1 = \varnothing$. Then, the set $\hat{\mathcal{Q}}_1$ is a $4$-minimal if and only if $Q_0$ contains no corner vertex.
    \end{itemize}
\end{lemma}
\begin{proof}
    If we have $x_1 = m-1$, $x_2=0$, $y_1 < y_2$, and $\hat{\mathcal{Q}}_i$ is a resolving set, since $v_1$ and $v_2$ are opposite-side boundary vertices, we may show that $Q_1 \cup \dots \cup Q_i$ is also a resolving set.

    We apply the grid isomorphism $T: P_m \square P_n \to P_n \square P_m$, $T(x,y) = (n-1-y, m-1-x)$, which is a rotation and reflection such that
    \begin{equation*}
        T(v_1) = (n-1-y_1, 0), \qquad T(v_2) = (n-1-y_2, m-1).
    \end{equation*}
    Since $y_1 < y_2 \implies n-1-y_1 > n - 1 - y_2$, $T(v_1)$ and $T(v_2)$ form an opposite-side boundary pair in the orientation required for Construction \ref{construction}. The pairs of the $F_{=1}$ type that are unresolved by $T(Q_1) = \{T(v_1), T(v_2)\}$ have respective horizontal and vertical bounds
    \begin{equation*}
        [n-1-y_2, n-1-y_1] \qquad [0, m-1].
    \end{equation*}
    Thus, $Z_1 \neq \varnothing$. Since $\hat{\mathcal{Q}}_i$ is resolving, the run does not terminate after $Q_1$, so $i \geq 2$. Now, for $Q_2 = \{v_3, v_4\}$, by Lemma \ref{lem: induction}, we have
    \begin{equation*}
        y_3 \geq y_2 \qquad y_4 \leq y_1 \qquad p < x_3 \qquad x_4 < q.
    \end{equation*}
    By Lemma \ref{lem: active-rectangle-update}, we find that the pairs of the $F_{=1}$ type that are unresolved by $T(Q_1 \cup Q_2)$ have respective horizontal and vertical bounds
    \begin{equation*}
        [n-1-y_2, n-1-y_1] \qquad [m-1-x_4, m-1-x_3].
    \end{equation*}
    By Lemma \ref{lem: induction}, these are the bounds for $Z_2$ under the grid isomorphism $T$.

    The same argument as above applies inductively to the remaining $Q_j$. We find that for every $j \geq 2$, the pairs of the $F_{=1}$ type that are unresolved by $T(Q_1 \cup \dots \cup Q_j)$ are exactly the ones in $T(Z_j)$.
    Suppose $Q_0, \dots, Q_i$ satisfy the recursive rules from Construction \ref{construction}, except that the prohibited terminal ordering is allowed, and suppose $Z_i = \varnothing$. Then, $T(Q_1 \cup \dots \cup Q_i)$ resolves every pair of the $F_{=1}$ type. Therefore, by Theorem \ref{thm: eliminating F_{<|1|}}, $T(Q_1 \cup \dots \cup Q_i)$ and therefore $Q_1 \cup \dots \cup Q_i$ resolves the entire grid. Thus, $\hat{\mathcal{Q}}_i$ is not a minimal resolving set.

    If $x_1 = m-1$, $x_2=0$, and $y_1 = y_2$, then the subset of vertices $\{v_0, v_1, v_2\} \subset \hat{\mathcal{Q}}_i$ is a resolving set by Theorem $2$ in \cite{andersen2016minimum}. Therefore, $\hat{\mathcal{Q}}_i$ is not a minimal resolving set.

    Finally, if $x_1 = m-1$, $x_2 = 0$, and $y_1 > y_2$, then $Z_1 = \mathcal{Z}(p, q, y_1, y_2)$, so $\hat{\mathcal{Q}}_1$ is a resolving set. 
    
    By Theorem $2$ in \cite{andersen2016minimum}, if $v_{-1} \in Q_0$ is a corner vertex, that is, $v_{-1} = (0, n-1)$, then the set $\{v_{-1}, v_1, v_2\}$ forms a $3$-minimal, so $\hat{\mathcal{Q}}_1$ is not a minimal resolving set. Likewise, if $v_0 \in Q_0$ is a corner vertex, that is, $v_0 = (m-1, 0)$, then the set $\{v_0, v_1, v_2\}$ forms a $3$-minimal, so $\hat{\mathcal{Q}}_1$ is not a minimal resolving set.
    
    If neither vertex in $Q_0$ is a corner vertex, then $\hat{\mathcal{Q}}_1$ contains no corner vertices and therefore contains no $2$-minimal by \cite{melter1984metric}. By Theorem $2$ in \cite{andersen2016minimum}, $\hat{\mathcal{Q}}_1$ contains no $3$-minimals, so it itself is a minimal resolving set.

    Therefore, if $x_1 = m-1$, $x_2 = 0$, and $y_1 > y_2$, then $\hat{\mathcal{Q}}_1$ is a minimal resolving set if and only if $Q_0$ contains no corner vertices.

\end{proof}

Define a region-valid prefix to be a sequence of pairs of vertices, $Q_0$, $Q_1$, \dots, $Q_i$ that satisfy the region-selection and boundary rules from Construction \ref{construction} without necessarily satisfying the $Q_1$ restrictions and termination rules. 

\begin{lemma} \label{lem: domination within R}
    Let $M$ be a minimal resolving set of $P_m \square P_n$ containing a region-valid prefix $Q_0, Q_1, \dots , Q_j$. Then, $M$ contains at most one vertex in any set $R_i^u$ or $R_i^b$ for every $0 \leq i \leq j$ for which $Z_i \neq \varnothing$. 
\end{lemma}
\begin{proof}
    We first consider the case where $i=0$. Suppose without loss of generality that the resolving set $M$ contains more than one vertex in $R_0^u$, so $M$ contains two vertices $w^1$ and $w^2$ in $R_0^u$. By Lemma \ref{lem: cutting of R_0}, neither vertex lies on a horizontal side, so $x_{w_1},x_{w_2}\leq p$.

    Without loss of generality, assume $y_{w_1}\leq y_{w_2}$. By Lemma 3.2, a vertex $w\in R_0^u$ resolves a pair
    \begin{equation*}
        \{(a,b),(a+t,b+t)\}\in Z_0
    \end{equation*}
    exactly when $y_w<b+t$. Hence every pair in $Z_0$ resolved by $w_2$ is also resolved by $w_1$. All pairs of type $F_{=1}$ outside $Z_0$ are already resolved by $Q_0$. Therefore, deleting $w_2$ from $M$ leaves every pair of type $F_{=1}$ resolved. By Theorem 3.4, $M\setminus\{w_2\}$ is still a resolving set, contradicting the minimality of $M$.
    
    The argument for $R_0^b$ is analogous. If $w_1,w_2\in M\cap R_0^b$, assume without loss of generality that $y_{w_1}\ge y_{w_2}$. By Lemma 3.2, a vertex $w\in R_0^b$ resolves
    \begin{equation*}
        \{(a,b),(a+t,b+t)\}\in Z_0
    \end{equation*}
    exactly when $y_w>b$. Thus every pair in $Z_0$ resolved by $w_2$ is also resolved by $w_1$, so $w_2$ may again be deleted without destroying resolvability. Therefore, $M$ contains at most one vertex from each of $R_0^u$ and $R_0^b$.
    
    Now, without loss of generality, suppose that for some $i$ prior to termination, the resolving set $M$ contains more than one vertex in $R_i^u$. Thus, $M$ contains some two vertices $w^1$ and $w^2$ in $R_i^u$. 

    If $i$ is odd, we have $R_i^u = \mathcal{R}(x_i^- + 1, x_i^+ -1, y_i^+, y_{i-1}^+ -1)$. Then, suppose that $x_{w^1} \geq x_{w^2}$. Then, by Lemma \ref{lem: induction}, we have $x_{w^1}, x_{w^2} \in [x_i^- +1, x_i^+ - 1] \subset (x_i^-, x_i^+)=(x_{2i-3}, x_{2i-2})$ and $y_{w^1}, y_{w^2} \geq y_i^+ = y_{2i}$.
    
    Then, since $Z_i = \mathcal{Z}(x_{2i-3}, x_{2i-2}, y_{2i-1}, y_{2i})$, by Lemma \ref{lem: active-rectangle-update}, the set of all pairs of vertices unresolved by $\hat{\mathcal{Q}}_i \cup \{w^1\}$ is $\mathcal{Z}(x_{w^1}, x_{2i-2}, y_{2i-1}, y_{2i})$
    and the set of all pairs of vertices unresolved by $\hat{\mathcal{Q}}_i \cup \{w^1, w^2\}$ is
    \begin{align*}
        \mathcal{Z}(x_{w^1}, x_{2i-2}, y_{2i-1}, y_{2i})
        \cap \mathcal{Z}(x_{w^2}, x_{2i-2}, y_{2i-1}, y_{2i})
        = \mathcal{Z}(x_{w^1}, x_{2i-2}, y_{2i-1}, y_{2i}).
    \end{align*}
    Every pair of vertices of the $F_{=1}$ form outside $Z_i$ is already resolved by $\hat{\mathcal{Q}}_i$, so after deleting $w^2$ from $M$, all pairs of vertices of the $F_{=1}$ form outside $Z_i$ are resolved by $\hat{\mathcal{Q}}_i \subseteq M$ while all pairs of the $F_{=1}$ form resolved by $w^2$ inside $Z_i$ are still resolved by $w^1$. 
    
    Thus, $M \setminus \{w^2\}$ still resolves all pairs of vertices of the $F_{=1}$ form, and by Theorem \ref{thm: eliminating F_{<|1|}}, it also resolves all pairs of the $F_{<|1|}$ form. Thus, if $w^1$ and $w^2$ are both in $M$, then $M$ is no longer minimal, as $M  \setminus \{w^2\}$ is still a resolving set.

    The remaining cases follow from the same argument, with the dominating vertex determined by the coordinate that changes the active bound. 
    
    If $i$ is odd and $w_1,w_2\in R_i^b$, label them so that $x_{w_1}\le x_{w_2}$. Lemma 4.2 shows that every pair in $Z_i$ resolved by $w_2$ is also resolved by $w_1$. If $i$ is even and $w_1,w_2\in R_i^u$, label them so that
    $y_{w_1}\le y_{w_2}$; again, every pair in $Z_i$ resolved by $w_2$ is also resolved by $w_1$. Finally, if $i$ is even and $w_1,w_2\in R_i^b$, label them so that $y_{w_1}\ge y_{w_2}$, in which case every pair in $Z_i$ resolved by
    $w_2$ is also resolved by $w_1$.
    
    In each case, all pairs of type $F_{=1}$ outside $Z_i$ are already resolved by $\hat{\mathcal{Q}}_i$. Thus the dominated vertex may be deleted from $M$ while leaving every pair of type $F_{=1}$ resolved. By Theorem 3.4, the remaining set is still resolving, contradicting the
    minimality of $M$. Therefore, $M$ contains at most one vertex in each region $R_i^u$ and $R_i^b$.
\end{proof}

\begin{lemma} \label{lem: forcing termination to hold}
    Let $M$ be a minimal resolving set containing a region-valid prefix $Q_0, \dots, Q_{i+1}$ where $Z_i \neq \varnothing$.
    
    We claim that if $i=0$ and $q - p = 1$, then we must have $y_2 \leq y_1$. If $i \geq 2$ is even, $Z_i \neq \varnothing$, and $x_{2i} - x_{2i-1} = 1$, then we must have $y_{2i+1} \geq y_{2i+2}$. If $i \geq 1$ is odd, $Z_i \neq \varnothing$, and $y_{2i} - y_{2i-1} = 1$, then we must have $x_{2i+1} \geq x_{2i+2}$.

    In every case, the necessary inequality implies $Z_{i+1} = \varnothing$. Thus, any minimal resolving set extending the construction prefix must terminate after the choice of $Q_{i+1}$ when the active rectangle $Z_i$ has a width or height of $1$.
\end{lemma}
\begin{proof}
    If $i=0$, $q - p = 1$, and $y_1 < y_2$, then $Z_1 = \mathcal{Z}(p,q, y_1,y_2)$, so we have the unresolved pair of vertices $\zeta_1^u = \{(p, y_2-1),(q, y_2)\}$. We find that $R_1^u = \mathcal{R}(p + 1, q - 1, y_2, n-2) = \varnothing$, as $q - p = 1 \implies p + 1 > q - 1$. Therefore, we have that $\Res(\zeta_1^u) \subseteq R_0^u \cup R_0^b$.
    
    $Q_1$ already contains a vertex from $R_0^u$ and a vertex from $R_0^b$, neither of which resolve $\zeta_1^u$, as $\zeta_1^u \in Z_1$. Therefore, any vertex resolving $\zeta_1^u$ must be a second vertex from $R_0^u$ or $R_0^b$, which would contradict Lemma \ref{lem: domination within R}. Therefore, $M$ would not be a minimal resolving set.
    
    If $i>0$ is even, suppose we have that $Z_i \neq \varnothing$, $x_{2i} - x_{2i-1} = 1$, and $y_{2i+1} < y_{2i+2}$. Then, we would have that $Z_{i+1} = \mathcal{Z}(x_{2i-1},x_{2i}, y_{2i+1},y_{2i+2})$.
    
    Then, we have that the pair of vertices $\zeta_{i+1}^u = \{(x_{2i-1}, y_{2i+2}-1),(x_{2i}, y_{2i+2})\}$ is unresolved. However, we also have that the set of vertices that we may add to our set to resolve $\zeta_{i+1}^u$, $R_{i+1}^u = \mathcal{R}(x_{2i-1} + 1, x_{2i} - 1, y_{2i+2}, y_{2i-2}-1)$, is empty, as $x_{2i} - x_{2i-1} = 1 \implies x_{2i-1} + 1 > x_{2i} - 1$.

    Then, we must have that
    \begin{align*}
        \Res(\zeta_{i+1}^u) \subseteq \bigcup_{0 \leq h \leq i} \left(\Res(\zeta_h^u) \cup \Res(\zeta_h^b) \right) = \bigcup_{0 \leq h \leq i}\left(R_h^u \cup R_h^b \right),
    \end{align*}
    with the last equality following from Lemma \ref{lem: induction}.

    However, our construction has already selected one vertex from each $R_h^u$ and $R_h^b$ for $0 \leq h \leq i$, none of which resolve $\zeta_{i+1}^u$. Thus, any additional vertex we add would be a second vertex in either $R_h^u$ or $R_h^b$ for some $0 \leq h \leq i$, contradicting Lemma \ref{lem: domination within R}. 
    
    Likewise, if $i$ is odd, suppose that we have that $Z_i \neq \varnothing$, $y_{2i} - y_{2i-1} = 1$, and $x_{2i+1} < x_{2i+2}$. Then, we would have that $Z_{i+1} = \mathcal{Z}(x_{2i+1},x_{2i+2}, y_{2i-1},y_{2i})$, and the pair of vertices $\zeta_{i+1}^u = \{(x_{2i+1}, y_{2i-1}),(x_{2i+1}+1, y_{2i})\}$ in $Z_{i+1}$ remains unresolved. 
    
    However, we have that the set of vertices we may add to our set to resolve $\zeta_{i+1}^u$, $R_{i+1}^u= \mathcal{R}(x_{2i-3}+1,x_{2i+1}, y_{2i-1}+1,y_{2i}-1)$, is empty, as $y_{2i} - y_{2i-1} = 1 \implies y_{2i-1}+1 > y_{2i}-1$. Then, we must have
    \begin{align*}
        \Res(\zeta_{i+1}^u) \subseteq  \bigcup_{0 \leq h \leq i} \left(\Res(\zeta_h^u) \cup \Res(\zeta_h^b) \right) = \bigcup_{0 \leq h \leq i}\left(R_h^u \cup R_h^b \right),
    \end{align*}
    with the last equality following from Lemma \ref{lem: induction}. However, our construction has already selected one vertex from each $R_h^u$ and $R_h^b$ for $0 \leq h \leq i$, none of which resolve $\zeta_{i+1}^u$. Thus, any additional vertex we add would be a second vertex in either $R_h^u$ or $R_h^b$ for some $0 \leq h \leq i$, contradicting Lemma \ref{lem: domination within R}. Thus, $Z_{i+1} = \varnothing$.

    Thus, $\hat{\mathcal{Q}}_{i+1}$ is resolving by Theorem \ref{thm: eliminating F_{<|1|}}. Since $\hat{\mathcal{Q}}_{i+1} \subseteq M$ and $M$ is minimal, $M = \hat{\mathcal{Q}}_{i+1}$, so $M$ must obey the terminal ordering as specified.
\end{proof}
In the construction, the choices for $Q_{i+1} = \{v_{2i+1}, v_{2i+2}\}$ described in the preceding lemmas are excluded. Together, these restrictions complete Construction \ref{construction}. 

\begin{theorem} \label{thm: soundness}
    Let $m, n \geq 3$. If a run of Construction \ref{construction} terminates after the selection of $Q_i$, then $\hat{\mathcal{Q}}_i$ is a minimal resolving set of $P_m \square P_n$ of size at least $4$.
\end{theorem}

\begin{proof}
    Since the construction terminates when $Z_i = \varnothing$ and $Z_i$ is the set of all pairs of vertices of the $F_{=1}$ form that are not resolved by $\hat{\mathcal{Q}}_i$, when $Z_i = \varnothing$, then $\hat{\mathcal{Q}}_i$ resolves all pairs of vertices of the $F_{=1}$ form. Thus, by Theorem \ref{thm: eliminating F_{<|1|}}, $\hat{\mathcal{Q}}_i$ resolves all pairs of vertices, so $\hat{\mathcal{Q}}_i$ must be a resolving set upon termination.

    It remains to show minimality for $\hat{\mathcal{Q}}_i$ by showing that every vertex uniquely resolves some pair of vertices.

    For each $Q_j = \{v_{2j-1}, v_{2j}\}$ where $1 \leq j \leq i$, by construction, one vertex resolves $\zeta_{j-1}^u$ and the other resolves $\zeta_{j-1}^b$. We show that $\zeta_{j-1}^u$ and $\zeta_{j-1}^b$ are uniquely resolved by these vertices.
    
    By construction, no vertex selected after $Q_j$ may resolve $\zeta_{j-1}^u$ or $\zeta_{j-1}^b$. No vertices in earlier pairs may resolve $\zeta_{j-1}^u$ or $\zeta_{j-1}^b$, as this contradicts the definition of $Z_{j-1}$. 
    
    Finally, since $v_1 \in R_0^b$ and $1 \leq y_1 \leq n-2$, $v_1$ does not resolve $\zeta_0^u$ by Lemma \ref{lem: char res vertices slope 1}, and likewise, $v_2$ does not resolve $\zeta_0^b$. From Lemma \ref{lem: char res vertices slope 1}, for all $h \geq 1$, all vertices in $R_h^u$ do not resolve $\zeta_h^b$ and all vertices in $R_h^b$ do not resolve $\zeta_h^u$, so vertices within each pair cannot resolve both $\zeta_{j-1}^u$ and $\zeta_{j-1}^b$.
    
    For $Q_0$, by Lemma \ref{lem: perpendicular pair restrictions}, we have two cases. For our first case, we have that $Q_1$ contains two boundary vertices perpendicular to $Q_0$ and that $\hat{\mathcal{Q}}_1$ forms a minimal resolving set, as by Lemma \ref{lem: perpendicular pair restrictions}, $Z_1 = \varnothing$ so the construction terminates after $Q_1$.

    For our second case, if $Q_0$ contains no corner vertices, the restrictions in Construction \ref{construction}  ensure that $Q_1$ contains at most one vertex on the two vertical sides. If $Q_0$ contains $(0, n-1)$, the $Q_1$ restrictions force $x_1 \neq m-1$, so $x = m-1$ is empty. If $Q_0$ contains $(m-1, 0)$, the $Q_1$ restrictions force $x_2 \neq 0$, so $x = 0$ is empty. By the region formulas in Lemma \ref{lem: induction}, every vertex selected after $Q_1$ is interior. Hence at least one of the edges $x=0$ and $x = m-1$ contains no selected vertex.

    If $\hat{\mathcal{Q}}_i$ does not contain a vertex on $x = 0$, then $v_{-1}$ uniquely resolves the pair $\{(0, n-2),(1, n-1)\}$ and $v_0$ uniquely resolves the pair $\{(0, 1),(1,0)\}$. We may check by Lemma \ref{lem: char res vertices slope 1} together with its image under the horizontal reflection $T(x,y) = (x, n-1-y)$ that $Q_1$ does not resolve either pair, $v_{-1}$ does not resolve $\{(0, 1),(1,0)\}$, and $v_0$ does not resolve $\{(0, n-2),(1, n-1)\}$.

    If $\hat{\mathcal{Q}}_i$ does not contain a vertex on $x = m-1$, then $v_{-1}$ uniquely resolves the pair $\{(m-2, n-1),(m-1, n-2)\}$ and $v_0$ uniquely resolves the pair $\{(m-2, 0),(m-1, 1)\}$. We may check by Lemma \ref{lem: char res vertices slope 1} together with its image under the horizontal reflection $T(x,y) = (x, n-1-y)$ that $Q_1$ does not resolve either pair, $v_{-1}$ does not resolve $\{(m-2, 0),(m-1, 1)\}$, and $v_0$ does not resolve $\{(m-2, n-1),(m-1, n-2)\}$.

    The stated pairs are unresolved by $Q_2 \cup \dots \cup Q_i$ because all vertices in $Q_2 \cup \dots \cup Q_i$ are interior, and by Lemma \ref{lem: char res vertices slope 1} together with its image under the horizontal reflection $T(x,y) = (x, n-1-y)$, the only vertices that resolve the stated pairs are boundary vertices. Therefore, every set of vertices $\hat{\mathcal{Q}}_i$ generated by the construction must be a minimal resolving set.
\end{proof}

\begin{theorem} \label{thm: completeness}
    The construction produces every minimal resolving set of $P_m \square P_n$ with $m,n\geq 3$ of size at least $4$.
\end{theorem}

\begin{proof}
    We will show that every minimal resolving set of $P_m \square P_n$ of size at least $4$ must be generated through Construction \ref{construction}.

    Choose any minimal resolving set of $P_m \square P_n$ of size at least $4$, and call it $M$. By Proposition $2$ in \cite{andersen2016minimum}, $M$ has at least one pair of opposite-side boundary vertices. Since $|M| \geq 4$, these vertices do not lie in the same row or column, as if they did, then $M$ would not be a minimal resolving set, as the vertices would form a resolving set of size $3$ with a third vertex in $M$ not in the same row or column. Such a third vertex must exist in $M$ since by Lemma $5$ in \cite{andersen2016minimum}, no minimal resolving set may contain $3$ vertices in the same row or column. By Proposition $1$ of \cite{andersen2016minimum}, $M$ contains at most $1$ corner vertex, so the third vertex together with the opposite-side vertex pair forms a $3$-minimal by Theorem $2$ of \cite{andersen2016minimum}, contradicting minimality.
    
    We may choose one of these pairs and after applying a rotation or reflection as necessary we may write them as $v_{-1} = (p, n-1)$ and $v_0 = (q, 0)$ with $q > p$. Then by Lemma \ref{lem: char nonres pairs}, the set of all pairs of vertices of the $F_{=1}$ form unresolved by $Q_0 = \{v_{-1}, v_0\}$ is $Z_0 = \mathcal{Z}(p,q,0,n-1)$.
    
    Then, we find that $M$ must contain some vertices that resolve $\zeta_0^u$ and $\zeta_0^b$. By Lemma \ref{lem: cutting of R_0}, these vertices cannot be on horizontal sides of the grid, so the sets of vertices that achieve this for $\zeta_0^u$ and $\zeta_0^b$ are $\mathcal{R}(0, p, 1, n-2)$ and $\mathcal{R}(q, m-1, 1, n-2)$ respectively. 
    
    The sets $\mathcal{R}(0, p, 1, n-2)$ and $\mathcal{R}(q, m-1, 1, n-2)$ are disjoint, so $M$ must contain at least one vertex from each. By Lemma \ref{lem: domination within R} with the disjoint sets, $M$ cannot contain more than one vertex from either. Thus, $M$ must contain exactly one vertex from $R_0^b$ and one vertex from $R_0^u$ satisfying the restrictions from Lemma \ref{lem: cutting of R_0}. We call these vertices $v_1$ and $v_2$ respectively, so we now have that $\hat{\mathcal{Q}}_1 \subseteq M$, where $Q_1 = \{v_1,v_2\}$.

    If $q - p = 1$, then Lemma \ref{lem: forcing termination to hold} applied at stage $i = 0$ gives $y_2 \leq y_1$, so our selected vertices must satisfy the coordinate inequalities. Consequently, $Z_1 = \varnothing$.

    Now, suppose that for some $i \geq 1$, we have selected some subset of $M$, $\hat{\mathcal{Q}}_i$.
    
    If $Z_i = \varnothing$, then $\hat{\mathcal{Q}}_i$ resolves all pairs of vertices of $P_m \square P_n$ and the construction terminates. By definition, $\hat{\mathcal{Q}}_i$ is a resolving set of $P_m \square P_n$. However, since $\hat{\mathcal{Q}}_i \subseteq M$ and $M$ is a minimal resolving set, we must have that $\hat{\mathcal{Q}}_i = M$, as otherwise, $M$ contains a proper subset of vertices that resolves $P_m \square P_n$, which contradicts minimality of $M$.
    
    If $Z_i \neq \varnothing$, by definition, $Z_i$ contains two pairs of vertices $\zeta_i^u$ and $\zeta_i^b$ that are unresolved by $\hat{\mathcal{Q}}_i$. Therefore, $M$ must contain additional vertices in $M \setminus \left(\hat{\mathcal{Q}}_i\right)$ that resolve $\zeta_i^u$ and $\zeta_i^b$. By Lemma \ref{lem: domination within R}, since $M$ contains one vertex in each $R_h^u$ and $R_h^b$ for $0 \leq h < i$, $M$ cannot contain any additional vertices in
    \begin{equation*}
        \bigcup_{0 \leq h <i} \left(R_h^u \cup R_h^b\right).
    \end{equation*}
    Therefore, by the final identity in Lemma \ref{lem: induction}, $M$ must contain at least one vertex in 
    \begin{equation*}
        \Res(\zeta_i^u) \setminus \bigcup_{0 \leq h <i} \left(R_h^u \cup R_h^b\right) = R_i^u
    \end{equation*}
    and one vertex in
    \begin{equation*}
        \Res(\zeta_i^b) \setminus \bigcup_{0 \leq h <i} \left(R_h^u \cup R_h^b\right) = R_i^b.
    \end{equation*}
    Since $R_i^u$ and $R_i^b$ are disjoint for all $i>0$ and by Lemma \ref{lem: domination within R} $M$ cannot contain more than one vertex in each region, $M$ must therefore contain one vertex in $R_i^u$ to resolve $\zeta_i^u$ and one vertex in $R_i^b$ to resolve $\zeta_i^b$.

    If $i$ is odd, then assign
    \begin{equation*}
        v_{2i+1} \in (M \setminus \left(\hat{\mathcal{Q}}_i\right)) \cap R_i^u,
        \qquad
        v_{2i+2} \in (M \setminus \left(\hat{\mathcal{Q}}_i\right)) \cap R_i^b.
    \end{equation*}
    If $i$ is even, then assign
    \begin{equation*}
        v_{2i+1} \in (M \setminus \left(\hat{\mathcal{Q}}_i\right)) \cap R_i^b,
        \qquad
        v_{2i+2} \in (M \setminus \left(\hat{\mathcal{Q}}_i\right)) \cap R_i^u.
    \end{equation*}
    Let $Q_{i+1} = \{v_{2i+1}, v_{2i+2}\}$, so $\hat{\mathcal{Q}}_{i+1} \subseteq M$. If the applicable bound difference of $Z_i$ is $1$, then by Lemma \ref{lem: forcing termination to hold}, the selected vertices $v_{2i+1}$ and $v_{2i+2}$ must satisfy the corresponding coordinate inequality, as otherwise, by Lemma \ref{lem: forcing termination to hold}, $M$ is not a minimal resolving set.
    
    We repeat with the next index, $i+1$. 
    
    Suppose that at every stage $i$, $\mathcal{Z}_i \neq \varnothing$. Then, the preceding argument gives two new vertices from $M \setminus \hat{\mathcal{Q}}_i$ at every stage, implying that $M$ contains infinitely many distinct vertices, which contradicts that $M \subseteq V_{m,n}$ is finite. Thus, we must have $\mathcal{Z}_i = \varnothing$ for some $i$, at which point the construction terminates and outputs
    \begin{equation*}
        \hat{\mathcal{Q}}_i = M,
    \end{equation*}
    as $\hat{\mathcal{Q}}_i$ resolves the grid and is a subset of the minimal resolving set $M$. By Lemma \ref{lem: perpendicular pair restrictions}, if the pair of vertices $Q_1$ does not satisfy the $Q_1$ restrictions of Construction \ref{construction}, then $M$ is not a minimal resolving set. This contradicts our assumption that $M$ is a minimal resolving set, so the pair of vertices $Q_1$ must obey the construction's restrictions. Therefore, the construction produces the chosen minimal resolving set $M$ of $P_m \square P_n$.
\end{proof}

Therefore, by Theorems \ref{thm: soundness} and \ref{thm: completeness}, the construction generates exactly the minimal resolving sets of $P_m \square P_n$ of cardinality at least $4$. The following corollary states a fact about opposite-side boundary vertices.
\begin{corollary} \label{coro: unique opposite side boundary vertex pair}
    Every minimal resolving set of the grid $P_m \square P_n$ of cardinality greater than $4$ contains a unique pair of boundary vertices on opposite sides.
\end{corollary}
\begin{proof}
    By Theorem \ref{thm: completeness}, every minimal resolving set of cardinality at least $4$ is generated by Construction \ref{construction}. Lemma \ref{lem: cutting of R_0} excludes additional vertices on the two sides containing the initial pair $Q_0$. Lemma \ref{lem: perpendicular pair restrictions} shows that if $Q_1$ contains an opposite-side pair on the two perpendicular sides, then either minimality is broken or the construction terminates after $Q_1$, producing a $4$-minimal. 
    
    Finally, the bounds in Lemma \ref{lem: induction} show that every vertex selected after $Q_1$ is interior. Therefore, an output of cardinality greater than $4$ contains exactly one opposite-side boundary pair, namely $Q_0$.
\end{proof}
The following previously known corollary from Adar and Epstein \cite{adarepstein2016alg} regarding the cardinalities of $k$-minimals follows naturally from the construction.
\begin{corollary} \label{coro: even size}
    For a grid $P_m \square P_n$, there are no minimal resolving sets of odd cardinality greater than $3$.
\end{corollary}
\begin{proof}
    By Theorem \ref{thm: completeness}, every $k$-minimal with $k \geq 4$ is generated by Construction \ref{construction}. The construction generates $k$-minimals using pairs of vertices, so it follows that there are no $k$-minimals of odd cardinality with $k > 3$.
\end{proof}

\section{Enumerating k-Minimals} \label{section: Enumeration}
From Corollary \ref{coro: even size}, all $k$-minimals of a grid $P_m \square P_n$ where $k > 3$ have even cardinality. In this section, we enumerate all $k$-minimals. Throughout, we use the convention that $\binom ab=0$ whenever $a<b$ or $b<0$. Let $N_k(m,n)$ denote the number of $k$-minimals of the grid $P_m \square P_n$.

The smallest possible value of $k$ is $2$, as this is the metric dimension of $P_m \square P_n$. All $2$-minimals must therefore be minimum resolving sets, and it is well-known for grids that these are exactly the $4$ pairs of two corner vertices sharing a side \cite{melter1984metric}.

\subsection{Enumerating 3-Minimals} \label{section: size 3 derivation}
For $P_m \square P_n$, the $3$-minimals are characterized by Theorem $2$ of Andersen et al. \cite{andersen2016minimum} to be all subsets $M$ of $3$ vertices of $P_m \square P_n$ such that
\begin{enumerate}
    \item $M$ has no more than one corner vertex.
    \item $M$ contains two boundary vertices on opposite sides.
    \item $M$ either has two vertices on the same line, $(i, j)$ and $(k, j)$ where $i \neq k$ and a third vertex $(p, q)$ where $i \leq p \leq k$ and $q \neq j$, or two vertices on the same line, $(i, j)$ and $(i, k)$ where $j \neq k$ and a third vertex $(p, q)$ where $j \leq q \leq k$ and $p \neq i$.
\end{enumerate}
We consider two cases of $3$-minimals. Either the $3$-minimal contains two opposite-side boundary vertices on the same line or it does not. First, suppose the opposite-side boundary vertices are on the same line, and by symmetry suppose they are vertical. There are $m-2$ ways for them to be vertical and non-corner vertices, then $(m-1)n$ ways to select the third vertex for a total of $(m-2)(m-1)n$ $3$-minimals. Likewise, there are $(n-2)(n-1)m$ such $3$-minimals whose opposite-side boundary vertices lie on the same horizontal line.

Now, suppose the $3$-minimal does not contain two opposite-side boundary vertices on the same line, and by symmetry suppose the two boundary vertices are on the horizontal sides of the grid. The third vertex must also be on one of the horizontal sides of the grid. There are $\binom{m}{3}$ ways to choose the $x$-coordinates of the vertices and $2$ ways to choose which side has two vertices. We subtract $2(m-2)$ for the cases with $2$ or more corner vertices, for a total of $2\binom{m}{3} - 2(m-2)$ $3$-minimals with two boundary vertices on the horizontal sides of the grid not on the same line. Likewise, there are $2\binom{n}{3}-2(n-2)$ such $3$-minimals with boundary vertices on the vertical sides of the grid.

The horizontal and vertical classes counted above are disjoint. If a  $3$-minimal belonged to both a horizontal and a vertical class, it would contain a horizontal opposite-side pair and a vertical opposite-side pair. As there are $3$ vertices, the two pairs would necessarily share a corner vertex, and Condition $3$ of Theorem $2$ of \cite{andersen2016minimum} forces one of the other vertices to also be a corner vertex, contradicting Condition $1$ of Theorem $2$ of \cite{andersen2016minimum}.

This yields a total of
\begin{equation*}
    2\binom{m}{3} + (n-2)(n-1)m + 2\binom{n}{3} + (m-2)(m-1)n - 2(m  + n - 4)
\end{equation*}
$3$-minimals of $P_m \square P_n$.

\subsection{Enumerating k-Minimals of All Even Cardinalities}
We call a run of Construction \ref{construction} admissible if it satisfies all restrictions of Construction \ref{construction}. An admissible initial pair is a choice of $Q_0$ permitted by Construction \ref{construction}.

We now count the construction runs of even cardinality $k = 2r$. We treat $r=2$ separately. For $r \geq 3$, a direct count would be cumbersome, as the regions from which successive vertex pairs are chosen are defined recursively. Our strategy is therefore to prove we may encode every admissible construction run by recording the horizontal and vertical coordinates of every pair of vertices subject to the Construction's inequalities.

First, fix an oriented initial pair $Q_0 = \{(p,n-1), (q, 0)\}$, where $p < q$. As the construction proceeds, every new pair added at each stage contributes horizontal and vertical coordinates that are subject to inequalities determined by the active rectangle's bounds. We record these inequalities in a horizontal profile $H_i$ and a vertical profile $V_i$. At a non-terminal stage of the construction, the two coordinates that become the new active bounds are inserted strictly between the previous bounds, and at the terminal stage, the final pair of coordinates appears in the opposite weak ordering. Thus, every run of Construction \ref{construction} can be encoded by two labeled mixed chains with entries satisfying $<$ and $\leq$ relations as prescribed by the construction.

For a fixed $Q_0$ and fixed horizontal coordinates of $Q_1$, Lemmas \ref{lem: recursive-profile invariant}, \ref{lem: bijection between collections}, and \ref{lem: distinctness of selected vertices} show that the horizontal and vertical profiles recover the recursive region-selection rules of Construction \ref{construction} and determine a valid labeled sequence of distinct vertices. Lemma \ref{lem: construction-run-injective} shows that for a fixed $Q_0$, distinct admissible construction runs produce distinct output vertex sets. Lemma \ref{lem: counting mixed chains} enumerates the number of horizontal and vertical profiles, which from Lemma \ref{lem: coordinate independence} may be combined independently.

Finally, we count the admissible choices of $Q_1$ in order to account for the initial restrictions of Construction \ref{construction}, multiply the two profile counts for a fixed $Q_0$, sum over all possible initial pairs $Q_0$, then sum over the four orientations of the construction to obtain our final enumeration. 

For the $r=2$ case, the construction terminates after $Q_1$ and therefore may be counted directly. The only additional class consists of $4$-minimals with two perpendicular pairs of opposite-side boundary vertices. By Lemma \ref{lem: perpendicular pair restrictions}, such outputs of Construction \ref{construction} with two perpendicular pairs of opposite-side boundary vertices must terminate at cardinality $4$.

\begin{definition} \label{def: labeled mixed chain}
    A labeled mixed chain is an ordered list $(c_1,c_2,\ldots,c_\ell)$ of labeled coordinate occurrences, together with a relation
    \begin{equation*}
        c_j\preceq_j c_{j+1}, \qquad \preceq_j\in \{\le,<\},
    \end{equation*}
    for every $1\leq j< \ell$.
\end{definition}

\begin{lemma} \label{lem: counting mixed chains}
    Let $A,B,L\in\mathbb Z$, where $A\leq B$ and $L\geq 1$. Let $S\subseteq \{1,2,\ldots,L-1\}$ and write $|S|=s$. The number of integer sequences $(a_1,a_2,\ldots,a_L)$ satisfying
    \begin{equation*}
        A \leq a_1 \leq a_2 \leq \dots \leq a_L\leq B, 
        \qquad 
        a_i<a_{i+1}
    \end{equation*}
    for every $i\in S$ is
    \begin{equation*}
        \binom{B-A+L-s}{L}.
    \end{equation*}
\end{lemma}

\begin{proof}
    For every $1\leq i\leq L$, define $c_i = |\{j\in S \mid j <i\}|$. That is, $c_i$ is the number of strict comparisons occurring before the $i$-th entry. Now define $b_i=a_i-c_i$. We find that $(b_1, b_2, \dots, b_L)$ is a weakly increasing $L$-tuple on $[A, B-s]$. Note that if $B-s < A$, then $[A, B-s] = \varnothing$, so there are no such $L$-tuples and the stated binomial coefficient is $0$. Thus, assume $B-s \geq A$. For any such weakly increasing $L$-tuple on $[A, B-s]$, we may construct the sequence $(a_1, a_2, \dots, a_L)$ by adding $c_i$ to every $b_i$ in the weakly increasing $L$-tuple to obtain our labeled mixed chain of length $L$ on $[A,B]$. Therefore, the number of labeled mixed chains of length $L$ chosen from interval $[A,B]$ is
    \begin{equation*}
        \binom{(B-A-s+1)+L-1}{L} = \binom{B-A+L-s}{L}.
    \end{equation*}
\end{proof}

For every non-terminal stage $i\ge 1$, define the active horizontal
bounds by
\begin{equation*}
    (x_i^-,x_i^+) =
    \begin{cases}
        x_{2i-3},x_{2i-2}, \qquad i\text{ odd},\\
        x_{2i-1},x_{2i}, \qquad i\text{ even},
    \end{cases}
\end{equation*}
and define the active vertical bounds by
\begin{equation*}
(y_i^-,y_i^+) =
    \begin{cases}
    y_{2i-1},y_{2i}, \qquad i\text{ odd},\\
    y_{2i-3},y_{2i-2}, \qquad i\text{ even}.
    \end{cases}
\end{equation*}

By Lemma~\ref{lem: induction}, $Z_i\neq\varnothing$ if and only if $x_i^-<x_i^+$ and $y_i^-<y_i^+$. Active bounds are defined only at non-terminal stages. For $i=1$, define $H_1=\varnothing$ and $V_1=(y_1<y_2)$. For the remainder of the profile construction, suppose the run reaches $Q_2$, so $Z_1 \neq \varnothing$. The pair $Q_2={v_3,v_4}$ satisfies
\begin{equation*}
    p+1 \leq x_3,x_4 \leq q-1
    \qquad \text{and} \qquad
    1 \leq y_4 \leq y_1<y_2 \leq y_3 \leq n-2.
\end{equation*}
If $Z_2 \neq \varnothing$, then $H_2=(x_3<x_4)$, and if $Z_2 = \varnothing$, then $H_2=(x_4\leq x_3)$. In either case, we have
\begin{equation*}
    V_2=(y_4 \leq y_1<y_2 \leq y_3).
\end{equation*}

Suppose $ i\geq 3$ is odd and $Z_i \neq \varnothing$. By the definitions of $R_i^u$ and $R_i^b$, $x_i^- < x_{2i+1},x_{2i+2} < x_i^+$.

If $Z_{i+1} \neq \varnothing$, replace the adjacent active pair $x_i^-<x_i^+$ with $x_i^- < x_{2i+1} < x_{2i+2} < x_i^+$.

If $Z_{i+1}=\varnothing$, replace the adjacent active pair with $x_i^- < x_{2i+2} \leq x_{2i+1} < x_i^+$.

The vertical region inequalities are $y_{2i-5}+1 \leq y_{2i+2} \leq y_{2i-1}$ and $y_{2i} \leq y_{2i+1} \leq y_{2i-4}-1$. Equivalently,
\begin{equation*}
    y_{2i-5} < y_{2i+2} \leq y_{2i-1} < y_{2i} \leq y_{2i+1} < y_{2i-4}.
\end{equation*}

Thus, $y_{2i+2}$ is inserted immediately before $y_{2i-1}$, and $y_{2i+1}$ is inserted immediately after $y_{2i}$. Now suppose that $i \geq 2$ is even and $Z_i \neq \varnothing$. By the definitions of $R_i^u$ and $R_i^b$, we have $y_i^- < y_{2i+1}$ and $y_{2i+2} < y_i^+$.

If $Z_{i+1} \neq \varnothing$, replace the adjacent active pair $y_i^-<y_i^+$ with $y_i^- < y_{2i+1} < y_{2i+2} < y_i^+$.

If $Z_{i+1}=\varnothing$, replace the adjacent active pair with $y_i^- < y_{2i+2} \leq y_{2i+1} < y_i^+$.

The horizontal region inequalities are $x_{2i-5}+1 \leq x_{2i+2} \leq x_{2i-1}$ and $x_{2i} \leq x_{2i+1} \leq x_{2i-4}-1$. Equivalently,
\begin{equation*}
    x_{2i-5} < x_{2i+2} \leq x_{2i-1} < x_{2i} \leq x_{2i+1} < x_{2i-4}.
\end{equation*}
Thus $x_{2i+2}$ is inserted immediately before $x_{2i-1}$, and
$x_{2i+1}$ is inserted immediately after $x_{2i}$.

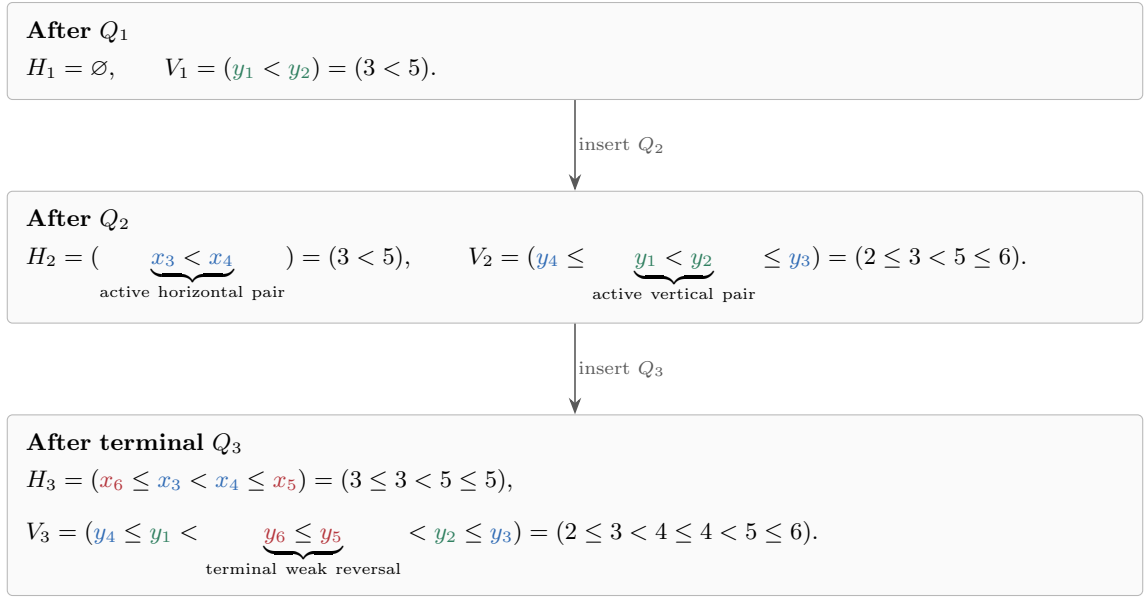
\begin{figure}[hbt!]
  \centering
  \begin{tikzpicture}[
      font=\small,
      profstage/.style={draw=black!28,rounded corners=2pt,fill=black!2,
        text width=0.88\textwidth,align=left,inner sep=7pt},
      flow/.style={-{Stealth[length=2.2mm]},semithick,black!65}
    ]
    \node[profstage] (S1) {
      \textbf{After $Q_1$}\par\smallskip
      $H_1=\varnothing,$\qquad
      $V_1=(\textcolor{pairgreen}{y_1}<
      \textcolor{pairgreen}{y_2})=(3<5)$.};

    \node[profstage,below=12mm of S1] (S2) {
      \textbf{After $Q_2$}\par\smallskip
      $H_2=(\underbrace{\textcolor{upperblue}{x_3}<
      \textcolor{upperblue}{x_4}}_{\text{active horizontal pair}})
      =(3<5)$, \qquad
      $V_2=(\textcolor{upperblue}{y_4}\le
      \underbrace{\textcolor{pairgreen}{y_1}<
      \textcolor{pairgreen}{y_2}}_{\text{active vertical pair}}
      \le\textcolor{upperblue}{y_3})=(2\le3<5\le6)$.};

    \node[profstage,below=12mm of S2] (S3) {
      \textbf{After terminal $Q_3$}\par\smallskip
      $H_3=(\textcolor{pairred}{x_6}\le
      \textcolor{upperblue}{x_3}<\textcolor{upperblue}{x_4}\le
      \textcolor{pairred}{x_5})=(3\le3<5\le5)$,\\[3mm]
      $V_3=(\textcolor{upperblue}{y_4}\le
      \textcolor{pairgreen}{y_1}<
      \underbrace{\textcolor{pairred}{y_6}\le
      \textcolor{pairred}{y_5}}_{\text{terminal weak reversal}}
      <\textcolor{pairgreen}{y_2}\le
      \textcolor{upperblue}{y_3})=(2\le3<4\le4<5\le6)$.};

    \draw[flow] (S1)--node[right,font=\scriptsize,fill=white,inner sep=1pt]
      {insert $Q_2$} (S2);
    \draw[flow] (S2)--node[right,font=\scriptsize,fill=white,inner sep=1pt]
      {insert $Q_3$} (S3);

    \node[below=5mm of S3,align=center,font=\scriptsize] {
      \textcolor{pairgreen}{Green}: coordinates of $Q_1$\qquad
      \textcolor{upperblue}{Blue}: coordinates of $Q_2$\qquad
      \textcolor{pairred}{Red}: coordinates of $Q_3$};
  \end{tikzpicture}
  \caption{Profile encoding for the construction run in
  Figure~\ref{fig:construction-run}.  At a non-terminal step, the two new
  active bounds enter in strict order.  At the terminal step, the new pair
  is inserted within the old active bounds in the opposite weak order, here
  $y_6\le y_5$, which records $Z_3=\varnothing$.}
  \label{fig:profiles}
\end{figure}

\begin{lemma} \label{lem: recursive-profile invariant}
    At stage $i \geq 2$, reached by the construction, the profiles $H_i$ and $V_i$ satisfy the following properties.
    
    \begin{enumerate}
        \item Every horizontal coordinate occurrence belonging to
        $Q_2,Q_3,\dots,Q_i$ appears exactly once in $H_i$.
        
        \item Every vertical coordinate occurrence belonging to $Q_1,Q_2,\ldots,Q_i$ appears exactly once in $V_i$.
        
        \item The active horizontal bounds occur as an adjacent strict pair in $H_i$, and the active vertical bounds occur as an adjacent strict pair in $V_i$.
        
        \item If $i \geq 3$ is odd, then the entries immediately outside the active vertical pair $y_{2i-1}<y_{2i}$ are $y_{2i-5}$ and $y_{2i-4}$.
        
        \item If $i=2$, the external bounds immediately outside the active horizontal pair $x_3 < x_4$ are $p = x_{-1}$ and $q = x_0$. If $i \geq 4$ is even, then the entries immediately outside the active horizontal pair $x_{2i-1} < x_{2i}$ are $x_{2i-5}$ and $x_{2i-4}$.

        \item The most recent insertion can be reversed uniquely, thereby recovering $H_{i-1}$ and $V_{i-1}$.
    \end{enumerate}
    Properties $1$, $2$, and $6$ hold at every stage with $i \geq 2$, including the terminal stage, while properties $3$, $4$, and $5$ hold only when $Z_i \neq \varnothing$.
\end{lemma}

\begin{proof}
    We proceed by induction on the stages $i \geq 2$ reached by the construction.
    
    At stage $i=2$, we have that $V_2=(y_4 \leq y_1<y_2\leq y_3)$ and 
    \begin{equation*}
        H_2=
        \begin{cases}
            (x_3<x_4) & Z_2 \neq \varnothing\\
            (x_4 \leq x_3) & Z_2 = \varnothing.
        \end{cases}
    \end{equation*}
    Every required coordinate occurrence appears exactly once. Deleting the occurrences belonging to $Q_2$ recovers $H_1 = \varnothing $ and $V_1 = (y_1 < y_2)$, so the most recent insertion is uniquely reversible.

    If $Z_2 \neq \varnothing$, then the active horizontal bounds $x_3 < x_4$ and the active vertical bounds $y_1 < y_2$ occur as adjacent strict pairs. Thus all properties applicable at stage $i=2$ hold.
    
    Assume that all assertions hold at a non-terminal stage $i \geq 2$. The next stage $i+1$ may be either terminal or non-terminal. 

    Suppose first that $i\geq 3$ is odd. In the horizontal profile, the active pair $x_{2i-3}<x_{2i-2}$ is replaced by $x_{2i-3}<x_{2i+1}<x_{2i+2}<x_{2i-2}$ if $Z_{i+1}\neq\varnothing$, and by $x_{2i-3}<x_{2i+2}\leq x_{2i+1}<x_{2i-2}$ if $Z_{i+1}=\varnothing$.

    In the vertical profile, the outer-neighbor property at stage $i$ and the recursive region inequalities give
    \begin{equation*}
        y_{2i-5} <y_{2i+2} \leq y_{2i-1} <y_{2i} \leq y_{2i+1}<y_{2i-4}.
    \end{equation*}
    Consequently, $y_{2i+2}$ and $y_{2i+1}$ have uniquely prescribed positions in $V_{i+1}$.
    
    It follows that every new coordinate occurrence appears exactly once. Since the occurrences are labeled, deleting the four occurrences belonging to $Q_{i+1}$ uniquely recovers $H_i$ and $V_i$.
    
    If $Z_{i+1}\neq\varnothing$, then $x_{2i+1}<x_{2i+2}$ is the new adjacent active horizontal pair, with $x_{2i-3}$ and $x_{2i-2}$ immediately outside it. The active vertical pair $y_{2i-1}<y_{2i}$ remains adjacent because the two new vertical occurrences were inserted immediately outside it. Thus all applicable assertions hold at stage $i+1$.
    
    Now suppose that $i$ is even. In the vertical profile, the active pair $y_{2i-3}<y_{2i-2}$ is replaced by $y_{2i-3}<y_{2i+1}<y_{2i+2}<y_{2i-2}$ if $Z_{i+1}\neq\varnothing$, and by $y_{2i-3}<y_{2i+2}\leq y_{2i+1}<y_{2i-2}$ if $Z_{i+1}=\varnothing$.
    
    If $i=2$, the recursive region inequalities give $p<x_6\leq x_3<x_4\leq x_5<q$. If $i\geq 4$, the outer-neighbor property at stage $i$ instead gives
    \begin{equation*}
        x_{2i-5} <x_{2i+2} \leq x_{2i-1} <x_{2i} \leq x_{2i+1}
        <x_{2i-4}.
    \end{equation*}
    Thus, in either case, the new horizontal coordinate occurrences have uniquely prescribed positions.
    
    Again, every new occurrence appears exactly once, and deleting the occurrences belonging to $Q_{i+1}$ uniquely recovers $H_i$ and $V_i$.
    
    If $Z_{i+1}\neq\varnothing$, then $y_{2i+1}<y_{2i+2}$ is the new adjacent active vertical pair, with $y_{2i-3}$ and
    $y_{2i-2}$ immediately outside it. The active horizontal pair $x_{2i-1}<x_{2i}$ remains adjacent. Hence all applicable assertions hold at stage $i+1$.
\end{proof}

\begin{lemma}
\label{lem: bijection between collections}
Fix $Q_0=\{(p,n-1),(q,0)\}$, where $p<q$, and fix $x_1\in[q,m-1]$, $x_2\in[0,p]$, and $(x_1,x_2)\neq(m-1,0)$. Fix $s\geq2$, prescribe that $Z_i\neq\varnothing$ for every $1 \leq i < s$, and prescribe whether $Z_s$ is empty or nonempty.

Recording $H_s,V_s$ gives a bijection between the following
collections:
\begin{enumerate}
\item labeled sequences $Q_1,Q_2,\ldots,Q_s$ with fixed horizontal coordinates $x_1,x_2$, satisfying the recursive region-selection and boundary rules and the prescribed status of $Z_s$, but not necessarily the additional forced-termination restrictions in
Construction \ref{construction};

\item pairs of integer realizations of the recursively defined
labeled mixed chains $H_s,V_s$, subject to $p+1\leq x_j\leq q-1$
for every entry of $H_s$, and $1\leq y_j\leq n-2$ for every entry of $V_s$.
\end{enumerate}
\end{lemma}

\begin{proof}
The forward map follows directly from the recursive definitions of
$H_s,V_s$ and the defining inequalities of the recursive regions.

Conversely, suppose that integer realizations of $H_s,V_s$ are given.
Every labeled horizontal coordinate occurrence belonging to $Q_2,Q_3,\ldots,Q_s$ appears exactly once in $H_s$, and every labeled vertical coordinate occurrence belonging to $Q_1,Q_2,\ldots,Q_s$
appears exactly once in $V_s$. The remaining horizontal coordinates
are the fixed values $x_1,x_2$. 

Thus the profiles determine every labeled vertex $v_j=(x_j,y_j)$
uniquely.

We verify the recursive region-selection rules by induction. The
external vertical bounds and the fixed values $x_1,x_2$ give $v_1\in R_0^b$ and $v_2\in R_0^u$.

Suppose that the reconstructed sequence satisfies the recursive rules through stage $i$, where $i<s$. Since $Z_i\neq\varnothing$,
Lemma \ref{lem: induction} gives the rectangular formulas for
$R_i^u$ and $R_i^b$. Each defining condition for these rectangles is a conjunction of one horizontal and one vertical interval condition.

Because all coordinates are integral, every strict profile inequality is equivalent to its corresponding closed integer bound. For example, $y_1 < y_2 \implies y_1 + 1 \leq y_2$.

Consequently, the labeled insertion relations in $H_s,V_s$ are
equivalent to the required membership conditions in $R_i^u,R_i^b$.
Thus the recursive region-selection rules hold at stage $i+1$.

At stage $i=1$, the horizontal profile is empty while the vertical profile is $y_1 < y_2$, so $Z_1 = \mathcal{Z}(p, q, y_1, y_2) \neq \varnothing$.

At every stage $i<s$ where $i \geq 2$, the profiles contain strict active pairs $x_i^-<x_i^+$ and $y_i^-<y_i^+$. By Lemma \ref{lem: induction}, $Z_i = \mathcal{Z}(x_i^-,x_i^+,y_i^-,y_i^+),$ so $Z_i \neq \varnothing$.

At stage $s$, the coordinate pair replaced at that stage is strict in the prescribed non-terminal case and weakly reversed in the prescribed terminal case. The corresponding formula from Lemma \ref{lem: induction} therefore gives the prescribed status of $Z_s$.

Since every coordinate occurrence is labeled and appears exactly once, the profiles and the fixed values $x_1,x_2$ determine the reconstructed labeled sequence uniquely. Therefore the correspondence is bijective.
\end{proof}

\begin{lemma} \label{lem: distinctness of selected vertices}
    Every labeled sequence occurring in collection $1$ of Lemma \ref{lem: bijection between collections} consists of pairwise distinct vertices.
\end{lemma}

\begin{proof}
    First, $Q_0$ contains two distinct vertices. The two vertices of $Q_1$ are distinct because
    \begin{align*}
        v_1 \in \{(x,y) \mid (x,y) \in R_0^b, y \in [1, n-2]\}, \qquad
        v_2 \in \{(x,y) \mid (x,y) \in R_0^u, y \in [1, n-2]\}
        \\
        \text{and } \{(x,y) \mid (x,y) \in R_0^b, y \in [1, n-2]\} \cap \{(x,y) \mid (x,y) \in R_0^u, y \in [1, n-2]\} = \varnothing.
    \end{align*}
    Neither vertex is equal to a vertex of $Q_0$ due to the $y$-coordinate restriction. 

    The vertices of $Q_2$ are distinct from previous vertices as they have horizontal coordinates strictly between $p$ and $q$ and all earlier vertices have horizontal coordinates $\leq p$ or $\geq q$. By Lemma \ref{lem: induction}, 
    \begin{equation*}
        v_3 \in R_1^u = \mathcal{R}(p+1, q-1, y_2, n-2), \qquad v_4 \in R_1^b = \mathcal{R}(p+1, q-1, 1, y_1),
    \end{equation*}
    so $y_4 \leq y_1 < y_2 \leq y_3$, as $Z_1 = \mathcal{Z}(p, q, y_1, y_2) \neq \varnothing$. Therefore, the vertices of $Q_2$, $v_3$ and $v_4$, are distinct from each other.

    Now consider a later pair $Q_{i+1}$. In the coordinate direction being shrunk at stage $i$, both newly selected coordinates lie strictly between the adjacent active bounds. By the recursive definition of the profile chains, no previously occurring coordinate lies strictly between those adjacent bounds. Consequently, neither new vertex can equal an earlier vertex.
    
    Moreover, the explicit formulas in Lemma~\ref{lem: induction} give $R_i^u\cap R_i^b=\varnothing$. Thus the two vertices within $Q_{i+1}$ are distinct from one another. Therefore all vertices in the labeled sequence are pairwise distinct.
    
\end{proof}

\begin{lemma} \label{lem: construction-run-injective}
    Fix an oriented initial pair $Q_0=\{(p,n-1),(q,0)\}$ with $p<q$. Let
    \begin{equation*}
        Q_0,Q_1,\ldots,Q_s  \qquad \text{and} \qquad Q_0,Q'_1,\ldots,Q'_t
    \end{equation*}
    be two admissible construction runs that terminate for the first time after $Q_s$ and $Q'_t$, respectively. If
    \begin{equation*}
        Q_0\cup Q_1\cup\dots\cup Q_s = Q_0\cup Q'_1\cup\dots\cup Q'_t,
    \end{equation*}
    then $s=t$ and $Q_i=Q'_i$ for every $1 \leq i \leq s$. Moreover, the labels of the two vertices within every pair are also uniquely determined.
    
    Consequently, for a fixed oriented pair $Q_0$, the map from
    admissible labeled construction runs to their output vertex sets is injective.
\end{lemma}

\begin{proof}
    Let 
    \begin{equation*}
        S = Q_0\cup Q_1\cup\dots\cup Q_s = Q_0\cup Q'_1\cup\dots\cup Q'_t
    \end{equation*}
    be the common output set. For each of the two runs separately, if its terminal index is $1$, pairwise distinctness follows directly from Construction \ref{construction}. If its terminal index is at least $2$, pairwise distinctness follows from Lemma \ref{lem: distinctness of selected vertices}. Hence the first run contains $2(s+1)$ distinct vertices and the second contains $2(t+1)$ distinct vertices. Since their output sets are equal, $2(s + 1) = |S| = 2(t+1) \implies s=t$.
    
    We now show inductively that the pairs $Q_i$ are uniquely determined by $S$ and $Q_0$.
    
    First consider $Q_1$. The pair $Q_1$ contains one vertex from $R_0^u$ and one vertex from $R_0^b$. No vertex of $Q_0$ belongs to either region, because both distinguished pairs $\zeta_0^u,\zeta_0^b$ are unresolved by $Q_0$.
    
    Moreover, every vertex selected after $Q_1$ belongs to some region $R_j^u$ or $R_j^b$ with $j\ge1$. By the definition of these regions, such a vertex does not resolve either $\zeta_0^u$ or $\zeta_0^b$. Hence no vertex selected after $Q_1$ belongs to $R_0^u\cup R_0^b$.
    
    It follows that $S\cap R_0^u$ and $S\cap R_0^b$ consist only of the two vertices selected in $Q_1$. By the boundary restriction, the admissible portions of $R_0^u$ and $R_0^b$ are disjoint, so the two vertices of \(Q_1\) are uniquely determined by $S$. Since the construction prescribes which region contains $v_1$ and which contains $v_2$, their individual labels are also uniquely determined. Therefore, $Q_1=Q'_1$.
    
    Now suppose inductively that $Q_j=Q'_j$ for every $0\le j\le i$. The common prefix $\hat{\mathcal{Q}}_i$
    uniquely determines $Z_i$, the distinguished pairs
    $\zeta_i^u,\zeta_i^b$, and consequently the regions
    $R_i^u,R_i^b$.
    
    No vertex in the existing prefix lies in $R_i^u \cup R_i^b$.
    Indeed, the pairs $\zeta_i^u,\zeta_i^b$ belong to $Z_i$, so neither is resolved by any vertex of $\hat{\mathcal{Q}}_i$.
    
    On the other hand, $Q_{i+1}$ contains exactly one vertex from $R_i^u$ and exactly one vertex from $R_i^b$. It remains to show that no later vertex belongs to either of these regions. Let $w$ be selected at a later stage $h>i$. Then
    \begin{equation*}
        w\in R_h^u\cup R_h^b.
    \end{equation*}
    By the definition of the regions at stage $h$, every such vertex is excluded from $\Res(\zeta_i^u) \cup \Res(\zeta_i^b)$.
    Since $R_i^u\subseteq \Res(\zeta_i^u)$ and $R_i^b\subseteq \Res(\zeta_i^b)$, we have that $w\notin R_i^u\cup R_i^b$.
    
    Therefore, $S\cap R_i^u$ and $S\cap R_i^b$ are precisely the two vertices of $Q_{i+1}$. The parity rule in
    Construction \ref{construction} determines which vertex is
    $v_{2i+1}$ and which is $v_{2i+2}$. Hence $Q_{i+1}$, including its labels, is uniquely determined by $S$ and the preceding prefix.
    
    By induction, $Q_i=Q'_i$ for every $1 \leq i \leq s$. Thus the two labeled runs are identical, and the map from admissible labeled runs beginning with $Q_0$ to their output vertex sets is injective.
\end{proof}

\begin{lemma} \label{lem: numbers of entries and strict comparisons}
    Suppose that the construction terminates for the first time after $Q_{r-1}$, where $r\geq 3$. Then $H_{r-1}$ has $2r-4$ entries and $r-3$ strict comparisons, while $V_{r-1}$ has $2r-2$ entries and $r-2$ strict comparisons.
\end{lemma}

\begin{proof}
    For $r=3$, the construction terminates after $Q_2$, and
    $H_2=(x_4 \leq x_3)$ while $V_2=(y_4 \leq y_1<y_2 \leq y_3)$.
    
    Thus $H_2$ has $2$ entries and $0$ strict comparisons, while
    $V_2$ has $4$ entries and $1$ strict comparison, as required.
    
    Assume now that $r \geq 4$. The first non-terminal horizontal shrinking stage $Q_2$ contributes one strict comparison. Every later non-terminal shrinking stage replaces one strict comparison by three strict comparisons and therefore increases the number of strict comparisons by $2$. 
    
    A terminal shrinking stage replaces one strict comparison by two strict comparisons and one weak comparison and therefore increases the number of strict comparisons by $1$. A non-shrinking insertion does not change the number of strict comparisons.
    
    Suppose first that $r-1=2 \ell$. Then the terminal stage is horizontal. The first horizontal shrinking stage contributes $1$, the $\ell-2$ intermediate horizontal shrinking stages contribute $2$ each, and the final horizontal shrinking stage contributes $1$. Hence the horizontal strict-comparison count is
    \begin{equation*}
    1+2(\ell-2)+1 = 2\ell-2 = r-3.
    \end{equation*}
    The vertical profile begins with the strict comparison $y_1<y_2$. The $\ell-1$ vertical shrinking stages are all non-terminal, so the vertical strict-comparison count is
    \begin{equation*}
        1+2(\ell-1) = 2\ell-1 = r-2.
    \end{equation*}
    
    Suppose next that $r-1=2 \ell+1$. Then the terminal stage is vertical. All horizontal shrinking stages are non-terminal, so their total contribution is
    \begin{equation*}
        1+2(\ell-1) = 2\ell-1 = r-3.
    \end{equation*}
    The vertical profile begins with one strict comparison. The first $\ell-1$ vertical shrinking stages contribute $2$ each, and the final vertical shrinking stage contributes $1$. Therefore the vertical strict-comparison count is
    \begin{equation*}
        1+2(\ell-1)+1 = 2\ell = r-2.
    \end{equation*}
    Finally, $H_{r-1}$ contains two horizontal coordinates from each of $Q_2,Q_3,\ldots,Q_{r-1}$, so $|H_{r-1}| = 2(r-2) = 2r-4$. Similarly, $V_{r-1}$ contains two vertical coordinates from each of $Q_1,Q_2,\ldots,Q_{r-1}$, so $|V_{r-1}| = 2(r-1) = 2r-2$.

\end{proof}

\begin{corollary} \label{coro: profile counts}
    Suppose $r \geq 3$ and the construction terminates immediately after $Q_{r-1}$. For fixed $Q_0=\{(p,n-1),(q,0)\}$ with $p<q$ and $q - p >1$, the number of horizontal profiles and the number of vertical profiles are
    \begin{equation*}
        \binom{q-p+r-3}{2r-4} \qquad \text{and} \qquad
        \binom{n+r-3}{2r-2}.
    \end{equation*}
\end{corollary}

\begin{proof}
    The horizontal profile is a prescribed mixed chain of length $2r-4$ in the interval $[p+1,q-1]$ with $r-3$ strict comparisons. By Lemma \ref{lem: counting mixed chains}, its number of realizations is
    \begin{equation*}
        \binom{(q-1)-(p+1)+(2r-4)-(r-3)}{2r-4} = \binom{q-p+r-3}{2r-4}.
    \end{equation*}
    The vertical profile is a prescribed mixed chain of length $2r-2$ in the interval $[1,n-2]$ with $r-2$ strict comparisons. Hence its number of realizations is
    \begin{equation*}
        \binom{(n-2)-1+(2r-2)-(r-2)}{2r-2} = \binom{n+r-3}{2r-2}.
    \end{equation*}
    Although the positions of the strict comparisons depend on the parity of the terminal stage, Lemma \ref{lem: counting mixed chains} shows that the number of realizations depends only on the total number of strict comparisons.
\end{proof}

\begin{lemma} \label{lem: coordinate independence}
Fix $s \geq 2$, an admissible initial pair $Q_0=\{(p,n-1),(q,0)\}$ with $p<q$ and a horizontal placement $x_1\in[q,m-1]$, $x_2\in[0,p]$, and $(x_1,x_2)\neq(m-1,0)$. Prescribe that $Z_i\neq\varnothing$ for every $1 \leq i < s$, and that $Z_s = \varnothing$.

Every realization of the corresponding horizontal profile may be
combined with every realization of the corresponding vertical profile. Each such pair determines a unique valid construction run
terminating for the first time after $Q_s$.
\end{lemma}

\begin{proof}
By Lemma \ref{lem: bijection between collections}, every pair of profile realizations determines a unique labeled sequence $Q_1,Q_2,\ldots,Q_s$ satisfying the recursive region-selection rules of the sets $Z_i$. In particular, $Z_i\neq\varnothing$ for $1 \leq i < s$ and $Z_s=\varnothing$.

It remains to verify the additional restrictions in
Construction \ref{construction}.

Since $s\geq2$, the vertical profile contains the strict comparison $y_1<y_2$, and its external bounds give $1\leq y_1<y_2\leq n-2$. We first verify the $Q_1$ restrictions. If neither vertex of $Q_0$ is a corner, the excluded placement $(x_1,x_2)=(m-1,0)$ is the only horizontal placement that would require $y_1>y_2$.

If $v^{-1}=(0,n-1)$, then $p=0$, so $x_2=0$. The exclusion $(x_1,x_2)\neq(m-1,0)$ therefore implies $x_1\neq m-1$. If $v^0=(m-1,0)$, then $q=m-1$, so $x_1=m-1$. The same exclusion implies
$x_2 \neq 0$. Thus all the $Q_1$ restrictions are satisfied.

Because $s \geq 2$, the horizontal profile contains $x_3,x_4$, both lying in the integer interval $[p+1,q-1]$. This interval must therefore be nonempty, so $q-p\geq2$. Hence the exceptional restriction for $q-p=1$ does not arise.

Now suppose that $i<s$ is odd and $y_i^+-y_i^-=1$. The selection of $Q_{i+1}$ changes only the active horizontal interval, so the active vertical interval at stage $i+1$ remains $[y_i^-,y_i^+]$.

If $s>i+1$, then stage $i+1$ is non-terminal, and selecting
$Q_{i+2}$ would require integer vertical coordinates strictly between $y_i^-$ and $y_i^+$. No such integer exists. Therefore $s=i+1$. Since stage $s$ is terminal, we therefore have $x_{2i+2}\leq x_{2i+1}$, which is exactly the required terminal ordering.

Similarly, suppose that $i \geq 2$ is even and $x_i^+-x_i^-=1$. The selection of $Q_{i+1}$ changes only the active vertical interval, so the active horizontal interval at stage $i+1$ remains $[x_i^-,x_i^+]$.

If $s>i+1$, then selecting $Q_{i+2}$ would require integer horizontal coordinates strictly between $x_i^-$ and $x_i^+$, which is impossible. Thus, $s=i+1$. The terminal profile therefore gives $y_{2i+2}\leq y_{2i+1}$, as desired.

Thus all additional restrictions in Construction \ref{construction} are satisfied. Lemma \ref{lem: distinctness of selected vertices} shows that all selected vertices are pairwise distinct. Therefore the reconstructed sequence is a unique valid construction run terminating for the first time after $Q_s$.
\end{proof}

\begin{lemma} \label{lem: count for a fixed initial pair}
    Fix an admissible initial pair $Q_0 = \{(p,n-1),(q,0)\}$ and $0\le p<q\le m-1$. For $r \geq 3$, the number of $2r$-minimals generated from $Q_0$ in this orientation is
    \begin{equation*}
        \left((p+1)(m-q)-1\right) \binom{q-p+r-3}{2r-4} \binom{n+r-3}{2r-2}.
    \end{equation*}
\end{lemma}

\begin{proof}
    By Construction \ref{construction}, the number of admissible
    horizontal placements of $Q_1$ is $(p+1)(m-q)-1$.

    If $q - p = 1$, our run terminates after $Q_1$ is chosen, so our expression for the number of $2r$-minimals generated from $Q_0$ with $r \geq 3$ is $0$. Thus, suppose $q - p > 1$, so by Corollary~\ref{coro: profile counts}, the number of horizontal and vertical profiles are
    \begin{equation*}
        \binom{q-p+r-3}{2r-4} \qquad \text{and} \qquad \binom{n+r-3}{2r-2}.
    \end{equation*}
    By Lemma~\ref{lem: coordinate independence}, these choices may be combined independently. The multiplication principle therefore gives
    \begin{equation*}
        \left((p+1)(m-q)-1\right) \binom{q-p+r-3}{2r-4} \binom{n+r-3}{2r-2}.
    \end{equation*}
    By Lemma \ref{lem: construction-run-injective}, distinct labeled runs with fixed $Q_0$ produce distinct vertex sets, so each choice as counted is unique.
\end{proof}

The following lemma covers the case where $r=2$, which may be considered directly.

\begin{lemma} \label{lem: r=2}
    For a fixed $Q_0 = \{(p, n-1), (q, 0)\}$ with $0 \leq p < q \leq m-1$, there are 
    \begin{equation*}
        ((p+1)(m-q)-1)\binom{n-1}{2}
    \end{equation*}
    $4$-minimals for which $Q_0$ is the unique opposite-side boundary vertex pair.
\end{lemma}
\begin{proof}
    For a fixed $Q_0 = \{(p, n-1), (q,0)\}$ where $p < q$, there are $(p+1)(m-q) - 1$ horizontal placements of $Q_1$, as there are $p+1$ choices for $v_2$, $m-q$ choices for $v_1$, and we subtract $1$ for the case in which both are boundary vertices. Since the construction terminates after $Q_1$, its vertical coordinates satisfy $1 \leq y_2 \leq y_1 \leq n-2$, which gives $\binom{n-1}{2}$ possibilities for the vertical coordinates. Therefore, for a fixed $Q_0$, there are 
    \begin{equation*}
        ((p+1)(m-q)-1)\binom{n-1}{2}
    \end{equation*}
    $4$-minimals for which $Q_0$ is the unique opposite-side boundary vertex pair.
\end{proof}
This is the $r=2$ specialization of the summand from Lemma \ref{lem: count for a fixed initial pair}. Thus, Lemma \ref{lem: evaluation of the initial-pair sum} uses this same summand for every $r \geq 2$ to count the $2r$-minimals with a unique opposite-side boundary vertex pair. The $4$-minimals with two perpendicular pairs of opposite-side boundary vertices are treated separately.
\begin{lemma}
    \label{lem: evaluation of the initial-pair sum}
    For every $r \geq 2$,
    \begin{equation*}
        \sum_{0\le p<q\le m-1} \left((p+1)(m-q)-1\right) \binom{q-p+r-3}{2r-4} = \frac{m+3r-1}{2r} \binom{m+r-2}{2r-1}.
\end{equation*}
\end{lemma}

\begin{proof}
    Set $d=q-p$. For fixed $d$, the variable $p$ ranges over $0\le p\le m-d-1$, and $q=p+d$.
    
    Therefore the left-hand side equals
    \begin{equation*}
        \sum_{d=1}^{m-1} \binom{d+r-3}{2r-4} \sum_{p=0}^{m-d-1} \left((p+1)(m-p-d)-1 \right).
    \end{equation*}
    Let $t=m-d$. Then
    \begin{equation*}
        \sum_{p=0}^{m-d-1} (p+1)(m-p-d) = \sum_{j=1}^{t} j(t+1-j) = \binom{t+2}{3} = \binom{m-d+2}{3}.
    \end{equation*}
    Hence the sum is
    \begin{equation*}
        S_r(m) = \sum_{d=1}^{m-1} \binom{m-d+2}{3} \binom{d+r-3}{2r-4}
        - \sum_{d=1}^{m-1} (m-d) \binom{d+r-3}{2r-4}.
    \end{equation*}
    
    For the first sum, Vandermonde's convolution gives
    \begin{equation*}
        \sum_{d=1}^{m-1} \binom{m-d+2}{3} \binom{d+r-3}{2r-4} = \binom{m+r}{2r}.
    \end{equation*}
    
    For the second sum, the hockey-stick identity gives
    \begin{align*}
        \sum_{d=1}^{m-1} (m-d) \binom{d+r-3}{2r-4} = \sum_{j=1}^{m-1} \sum_{d=1}^j \binom{d+r-3}{2r-4}= \sum_{j=1}^{m-1} \binom{j+r-2}{2r-3} =\binom{m+r-2}{2r-2} \\
        \implies 
        S_r(m) = \binom{m+r}{2r} - \binom{m+r-2}{2r-2} 
        = \binom{m+r-2}{2r-2} \left(\frac{(m+r)(m+r-1)} {(2r)(2r-1)} - 1 \right)
        \\
        =
        \binom{m+r-2}{2r-2} \frac{(m+r)(m+r-1) -(2r)(2r-1)}{(2r)(2r-1)}.
    \end{align*}
    The numerator factors as $(m+r)(m+r-1) -(2r)(2r-1) = (m-r)(m+3r-1)$, so
    \begin{equation*}
        S_r(m) = \binom{m+r-2}{2r-2} \frac{(m-r)(m+3r-1)}{(2r)(2r-1)}.
    \end{equation*}
    
    Finally,
    \begin{equation*}
        \binom{m+r-2}{2r-1} = \frac{m-r}{2r-1} \binom{m+r-2}{2r-2}.
    \end{equation*}
    Therefore
    \begin{equation*}
        S_r(m) = \frac{m+3r-1}{2r}\binom{m+r-2}{2r-1}.
    \end{equation*}
\end{proof}

\begin{lemma} \label{lem: one horizontal orientation}
    Let $r \geq 3$. The number of $2r$-minimal resolving sets generated in one fixed
    horizontal orientation is
    \begin{equation*}
        A_{2r}(m,n) = \frac{m+3r-1}{2r} \binom{m+r-2}{2r-1} \binom{n+r-3}{2r-2}.
    \end{equation*}
\end{lemma}

\begin{proof}
    By Lemma~\ref{lem: count for a fixed initial pair},
    \begin{equation*}
        A_{2r}(m,n) = \binom{n+r-3}{2r-2}\cdot \sum_{0\le p<q\le m-1} \left((p+1)(m-q)-1\right) \binom{q-p+r-3}{2r-4}.
    \end{equation*}
    Note that the forbidden two-corner pair where $p = 0$ and $q = m-1$ contributes $0$. We may sum over all possible $Q_0$, as by Corollary \ref{coro: unique opposite side boundary vertex pair}, every output has a unique opposite-side boundary vertex pair, so different $Q_0$ pairs result in different vertex sets.
    
    Applying Lemma \ref{lem: evaluation of the initial-pair sum} gives
    \begin{equation*}
        A_{2r}(m,n) = \frac{m+3r-1}{2r} \binom{m+r-2}{2r-1} \binom{n+r-3}{2r-2}.
    \end{equation*}
\end{proof}

\begin{theorem} \label{thm: enumeration}
    Let $m,n \geq 3$ and $r \geq 3$. Then the number of minimal resolving sets of cardinality $2r$ in $P_m\square P_n$ is
    \begin{equation*}
        N_{2r}(m,n) = \frac{1}{r} \left[(m+3r-1) \binom{m+r-2}{2r-1} \binom{n+r-3}{2r-2} + (n+3r-1) \binom{n+r-2}{2r-1} \binom{m+r-3}{2r-2} \right].
    \end{equation*}
\end{theorem}

\begin{proof}
    By Theorem \ref{thm: completeness}, the four oriented construction classes are exhaustive. By Corollary \ref{coro: unique opposite side boundary vertex pair}, every $2r$-minimal with $r \geq 3$ has a unique pair of opposite-side boundary vertices. The pair of sides determines whether the orientation is horizontal or vertical, while the relative order of the coordinates determines the reflective orientation. Therefore, the four oriented construction classes are disjoint.
    
    The two horizontal orientations contribute $2A_{2r}(m,n)$, while the two vertical orientations contribute $2A_{2r}(n,m)$.
    Therefore, $N_{2r}(m,n) = 2A_{2r}(m,n) + 2A_{2r}(n,m)$.
    
    By Lemma \ref{lem: one horizontal orientation},
    \begin{equation*}
        N_{2r}(m,n) = 2\left[\frac{m+3r-1}{2r}\binom{m+r-2}{2r-1} \binom{n+r-3}{2r-2} \right] + 2 \left[\frac{n+3r-1}{2r} \binom{n+r-2}{2r-1} \binom{m+r-3}{2r-2} \right].
    \end{equation*}
    Simplifying gives
    \begin{equation*}
        N_{2r}(m,n) = \frac{1}{r}\left[(m+3r-1)\binom{m+r-2}{2r-1}
        \binom{n+r-3}{2r-2} + (n+3r-1)\binom{n+r-2}{2r-1}\binom{m+r-3}{2r-2}\right].
    \end{equation*}
\end{proof}

The binomial-coefficient convention accounts for
infeasible grid dimensions. For instance,
\begin{equation*}
    \binom{m+r-2}{2r-1}=0
\end{equation*}
whenever $m<r+1$. Thus, no separate case is required when a grid is too narrow to support a $2r$-minimal in the corresponding orientation.

The following corollary combines the results of Corollary \ref{coro: even size}, Subsection \ref{section: size 3 derivation}, and Theorem \ref{thm: enumeration} to provide a complete enumeration of the minimal resolving sets of rectangular grid $P_m \square P_n$.

\begin{corollary} \label{coro: finalenum}
    Let $m, n \geq 3$. For every positive integer $k$, $P_m \square P_n$ has $k$-minimals if and only if $k \in \{2,3\} \cup \{2r \mid 2 \leq r \leq \min(m,n)-1 \}$. Then, the number of $k$-minimals, $N_k(m,n)$, is 
    \begin{align*}
      N_k(m,n) = \begin{dcases}
        4 & k=2\\
        2\binom{m}{3} + (n-2)(n-1)m + 2\binom{n}{3} + (m-2)(m-1)n - 2(m  + n - 4) & k=3
        \\
        \frac{1}{2}\left[(m+5)\binom{m}{3}\binom{n-1}{2} + (n+5)\binom{n}{3} \binom{m-1}{2}\right] + 2\binom{m-2}{2} \binom{n-2}{2} & k=4\\
        \frac{1}{r}\left[(m+3r-1)\binom{m+r-2}{2r-1}
        \binom{n+r-3}{2r-2} + (n+3r-1)\binom{n+r-2}{2r-1}\binom{m+r-3}{2r-2}\right] & k = 2r\\
        0 & \text{else,}
        \end{dcases}
    \end{align*}
    where $3 \leq r \leq \min(m,n) -1$.
\end{corollary}
\begin{proof}
    By Melter and Tomescu, the $2$-minimals for grids are the minimum resolving sets, which are the $4$ pairs of two corner vertices sharing a side \cite{melter1984metric}. The expression for $k=3$ is from subsection \ref{section: size 3 derivation} and is positive for all $m,n \ge 3$, so $3$-minimals always exist.

    For $k = 4$, Lemma \ref{lem: r=2} counts the $4$-minimals for which the chosen $Q_0$ is the unique opposite-side boundary pair. By Lemma \ref{lem: evaluation of the initial-pair sum}, summing over all $Q_0$ for one fixed horizontal orientation gives
    \begin{equation*}
        \binom{n-1}{2}\frac{m+5}{4}\binom{m}{3}.
    \end{equation*}
    The two horizontal orientations and two vertical orientations respectively contribute 
    \begin{equation*}
        \frac{m+5}{2}\binom{m}{3}\binom{n-1}{2} \qquad \text{and} \qquad \frac{n+5}{2}\binom{n}{3}\binom{m-1}{2}.
    \end{equation*}
    Since every set counted here has a unique opposite-side boundary pair, these four orientations are disjoint.
    
    The preceding sum omits exactly the $4$-minimals with two perpendicular opposite-side boundary pairs. By Lemma \ref{lem: perpendicular pair restrictions}, these are exactly the exceptional minimal resolving sets with $4$ boundary vertices, none being corners.
    
    We have two cases depending on whether $p < q$ or $p > q$. Take $p < q$. Then, we have $\binom{m-2}{2}$ ways of choosing the two vertices on the horizontal boundaries and $\binom{n-2}{2}$ ways of choosing the two vertices on the vertical boundaries. The $p>q$ case is analogous, resulting in a total of 
    \begin{equation*}
        2\binom{m-2}{2} \binom{n-2}{2}
    \end{equation*}
    additional minimal resolving sets for $k=4$, or a total of
    \begin{equation*}
        \frac{1}{2}\left[(m+5)\binom{m}{3}\binom{n-1}{2} + (n+5)\binom{n}{3} \binom{m-1}{2}\right] + 2\binom{m-2}{2} \binom{n-2}{2}
    \end{equation*}
    $4$-minimals. This expression is positive for all $m,n \geq 3$, so $4$-minimals exist for every rectangular grid with $m,n \geq 3$.
    
    For even values of $k > 4$, when $k=2r$ and $r \geq 3$, we have
    \begin{align*}
        \binom{m+r-2}{2r-1} > 0 \iff \binom{m+r-3}{2r-2} >0 \iff m - 1 \geq r,
        \\
        \binom{n+r-2}{2r-1}>0 \iff \binom{n+r-3}{2r-2} >0 \iff n-1 \geq r,
    \end{align*}
    so $N_{2r}(m,n) >0 \iff r \leq \min(m,n)-1$. For odd values of $k > 4$, $N_k(m,n) = 0$ by Corollary \ref{coro: even size}.
\end{proof}

\section{Conclusion and Further Directions}

We have completed the characterization and enumeration of inclusion-minimal resolving sets for rectangular grid graphs. The recursive construction generates exactly the minimal resolving sets of cardinality at least $4$, while the resulting profile encoding yields closed formulas for every possible cardinality. Together with the known descriptions of $2$-minimals and $3$-minimals, this determines the entire minimal-resolving-set structure of $P_m\square P_n$.

Conceptually, the construction shows that a global metric condition can be governed by a nested sequence of local geometric constraints. This mechanism may extend beyond rectangular grids. As metric dimension has been investigated for Cartesian products of graphs \cite{caceres2007cartesian}, it would be interesting to determine which Cartesian products admit exact structural characterizations and closed enumerations of their minimal resolving sets. Natural next cases, also suggested by Andersen et al. \cite{andersen2016minimum}, include cylindrical grids $P_m \square C_n$ and toroidal grids $C_m \square C_n$, where the disappearance of one or both pairs of boundary sides prevents direct application of the construction provided here for rectangular grids $P_m \square P_n$. Higher-dimensional grids are another natural case to examine, for which an extension of the construction may be applicable.

\section*{Acknowledgments}

The author would like to thank Jesse Geneson for his guidance, feedback, and support throughout this project. The author also thanks the MIT PRIMES-USA program for fostering this research experience and providing a supportive environment in which this research was conducted.

OpenAI’s ChatGPT (GPT-5.6) was used as an auxiliary tool for manuscript revision. It was used as a referee to identify possible grammatical, stylistic, organizational, and notational issues. The research question, central construction, mathematical results, proofs, enumeration arguments, computations, and conclusions were developed by the author under the guidance of the supervising instructor. No AI-generated citation, mathematical statement, or code was incorporated without manual verification.

\printbibliography
\addcontentsline{toc}{section}{References}
\vspace{1em}
\noindent
\textsc{PRIMES USA, MIT, Cambridge, MA, 02139} \\
\textit{Email address:} \texttt{cocoviolin08@gmail.com}

\end{document}